%% file: k_regularity_export_arxiv.tex
\documentclass[12pt]{amsart}
\usepackage{etex}
\reserveinserts{28}

\usepackage[T1]{fontenc}
\usepackage[cp1250]{inputenc}
\usepackage[margin=1in, marginpar=0.8in]{geometry}
\usepackage{color}
\usepackage[table]{xcolor}
\usepackage[centertags]{amsmath}
\usepackage{amsfonts}
\usepackage{amsthm}
\usepackage{amsgen}
\usepackage{amssymb}
\usepackage{amssymb}
\usepackage{amscd}
\usepackage{enumerate}
\usepackage{color}
\usepackage{longtable}
\usepackage{multicol}
\usepackage{graphicx}
\usepackage{colordvi}
\usepackage{todonotes}
\usepackage{tikz-cd}
\usepackage{booktabs}

\usepackage[all,  colour, line, dvips]{xy}
\usepackage{textcomp}
\newlength{\defbaselineskip} 
\usepackage[normalem]{ulem}
\theoremstyle{plain}
\newtheorem{thm}{Theorem}[section]
\newtheorem{cor}[thm]{Corollary}
\newtheorem{con}[thm]{Conjecture}
\newtheorem{df}[thm]{Definition}
\newtheorem{propdef}[thm]{Proposition-Definition}
\newtheorem{lema}[thm]{Lemma}
\newtheorem{prop}[thm]{Proposition}
\newtheorem{obs}[thm]{Proposition}
\newtheorem{exm}[thm]{Example}

\newtheorem{question}[thm]{Question}

\newtheorem{fact}[thm]{Fact}

\newtheorem{rem}[thm]{Remark}
\newtheorem{pr}{Algorithm}
\theoremstyle{definition} 
\theoremstyle{definition} \newtheorem{ex}{Example}[section] %

 \numberwithin{equation}{section}

\def\p{\mathbb{P}}

\def\n{\mathbb{N}}

\def\fa{\begin{fact}}
\def\kfa{\end{fact}}

\newcommand{\fromto}[2]{#1, \dotsc, #2}
\newcommand{\setfromto}[2]{\set{\fromto{#1}{#2}}}

\renewcommand\theenumi{(\roman{enumi})}
\renewcommand\labelenumi{\theenumi}

\def\ccC{\mathcal{C}}

\def\ccE{\mathcal{E}}

\def\ccH{\mathcal{H}}

\def\R\mathcal{R}

\newcommand{\set}[1]{\left\{#1\right\}}

\DeclareMathOperator{\Spec}{Spec}

\DeclareMathOperator{\rad}{rad}
\DeclareMathOperator{\length}{length}
\newcommand{\len}{\length}
\DeclareMathOperator{\Hilb}{Hilb}
\DeclareMathOperator{\HilbGor}{Hilb^{Gor}}

\DeclareMathOperator{\HilbP}{HilbP}
\DeclareMathOperator{\HilbPGor}{HilbP^{Gor}}
\DeclareMathOperator{\HilbSm}{Hilb^{sm}}
\DeclareMathOperator{\HilbPSm}{HilbP^{sm}}

\def\p{\mathbb{P}}

\def\ob{\begin{obs}}
\def\kob{\end{obs}}
\def\dow{\begin{proof}}
\def\kdow{\end{proof}}

\def\tw{\begin{thm}}
\def\ktw{\end{thm}}
\def\hip{\begin{con}}
\def\khip{\end{con}}
\def\lem{\begin{lema}}
\def\klem{\end{lema}}
\def\ex{\begin{exm}}
\def\prog{\begin{pr}}
\def\kprog{\end{pr}}
\def\wn{\begin{cor}}
\def\kwn{\end{cor}}
\def\uwa{\begin{rem}}
\def\kuwa{\end{rem}}
\def\kex{\end{exm}}
\def\dfi{\begin{df}}
\def\kdfi{\end{df}}
\definecolor{zielony}{rgb}{0.5, 0.9, 0.1}
\definecolor{czerwony}{rgb}{0.9, 0.2, 0.1}
\definecolor{niebieski}{rgb}{0.3, 0.1, 0.9}

\newenvironment{red}{\color{czerwony}}{}
\newcommand{\bred}{\begin{red}}
 \newcommand{\ered}{\end{red}}

\title{Constructions of $k$-regular maps using finite local schemes}
\date{26th September 2016}
\author[J.~Buczy\'nski]{Jaros\l{}aw Buczy\'nski}
\thanks{J.~Buczy\'nski is supported by a grant Iuventus Plus of the Polish Ministry of Science, project 0301/IP3/2015/73, and by a scholarship of Polish Ministry of Science}
 \address{Jaros\l{}aw Buczy\'nski\\
Faculty of Mathematics, Computer Science and Mechanics\\
University of Warsaw\\
ul. Banacha 2\\
02-097 Warszawa\\
Poland\\
and
Institute of Mathematics of Polish Academy of Sciences\\
  ul. \'Sniadeckich 8\\
  00-656 Warszawa, Poland
 }
 \email{jabu@mimuw.edu.pl}

\author[T.~Januszkiewicz]{Tadeusz Januszkiewicz}
\thanks{T.~Januszkiewicz is supported by Polish National Science Center, project 2012/06/A/ST1/00259}
 \address{Tadeusz Januszkiewicz\\
Institute of Mathematics of Polish Academy of Sciences\\
  ul. \'Sniadeckich 8\\
  00-656 Warszawa, Poland}
 \email{tjan@impan.pl}

\author[J.~Jelisiejew]{Joachim Jelisiejew}
\thanks{J.~Jelisiejew is supported by Polish National Science Center, project 2014/13/N/ST1/02640}
\address{Joachim Jelisiejew\\
Faculty of Mathematics, Computer Science and Mechanics\\
University of Warsaw\\
ul. Banacha 2\\
02-097 Warszawa, Poland}
\email{j.jelisiejew@mimuw.edu.pl}

\author[M.~Micha{\l}ek]{Mateusz Micha{\l}ek}
 \thanks{M.~Micha{\l}ek is supported by a grant Iuventus Plus of the Polish Ministry of Science, project 0301/IP3/2015/73}
 \address{Mateusz Micha{\l}ek\\
 Freie Universit\"at\\
 Arnimallee 3\\
 14195 Berlin, Germany\\
 and Institute of Mathematics of Polish Academy of Sciences\\
 ul. \'Sniadeckich 8\\
 00-956 Warszawa, Poland}
 \email{wajcha2@poczta.onet.pl}

\newcommand{\RR}{\mathbb{R}}
\newcommand{\CC}{\mathbb{C}}
\newcommand{\FF}{\mathbb{F}}
\newcommand{\BB}{\mathbb{B}}
\newcommand{\PP}{\mathbb{P}}
\newcommand{\ZZ}{\mathbb{Z}}
\DeclareMathOperator{\Sym}{Sym}
\newcommand{\gota}{\mathfrak{a}}

\definecolor{Zielony}{RGB}{78,104,18} 
\definecolor{JasnoZielony}{RGB}{195,204,110} 

\usepackage[pdfborder={0 0 0}]{hyperref}

\begin{document}
\begin{abstract}
	A continuous map $\RR^m\rightarrow \RR^N$ or $\CC^m\rightarrow \CC^N$ is called $k$-regular
	if the images of any $k$ points are linearly independent.
	Given integers $m$ and $k$  a problem going back to Chebyshev and  Borsuk is
	to determine the minimal value of $N$ for which such maps exist.
	The methods of algebraic topology provide lower bounds for $N$,
	however there are very few results on the existence of such maps for particular values $m$  and $k$.
	Using the methods of algebraic geometry we construct $k$-regular maps.
	We relate the upper bounds on $N$ with the dimension of the locus of certain Gorenstein schemes
	  in the punctual Hilbert scheme.
	The computations of the dimension of this family is explicit for $k\leq 9$,
          and we provide explicit examples for $k\leq 5$.
	We also provide upper bounds for arbitrary $m$ and $k$.
\end{abstract}

\maketitle

\section{Introduction}

\subsection{\texorpdfstring{$k$}{k}-regularity}
Given an $m$-dimensional manifold $M$, we say that a continuous map $f\colon M \to \RR^N$
  is  \emph{$k$-regular}  if the images of any $k$ distinct points of $M$ span a $k$-dimensional linear subspace.
\begin{exm}\label{exam_vandermonde}
   Consider the map $\RR^1\to \RR^k$ given by
   \[
      t \mapsto (1, t, t^2, \dotsc, t^{k-1}).
   \]
The Vandermonde determinant shows that this map is $k$-regular.
\end{exm}
For more precise  definition and several variants of it,  see Section~\ref{sect_independence}.
The definition of \mbox{$k$-regular} maps in this form was introduced by Borsuk \cite{Borsuk}
in 1957. An equivalent concept of $k$-interpolating-at-arbitrary-nodes was
known since Chebyshev with contributions by Haar \cite{haar__kreg}, Kolmogorov \cite{kolmogorov__kreg} and others. However only after Borsuk's paper the topologists got interested in this subject.

 Embeddings are precisely examples of affinely $2$-regular maps, as the image
of any two points of the manifold spans an affine line --- a maximal dimensional affine space.
Although the problem of embeddings of affine spaces is trivial, it turned out that $k$-regular maps from affine spaces are very interesting.
The problem when such maps exist attracted the attention of many algebraic topologists
 \cite{kreg1,kreg2,kreg3,kreg4,kreg5,kreg6}.

The existence of $k$-regular maps is also important in approximation theory,
by their connection with \emph{interpolation spaces} \cite{kreg3,wulbert1999interpolation, interpolation__Shekhtman__poly, interpolation__Shekhtman__ideal}. Let $V$ be a subspace of
continuous functions on $\RR^m$, then we say that $V$ is a
$k$-\emph{interpolation space} if for every set $\mathcal{Z}$ of $k$ distinct points on
$\RR^m$ and every function $f\colon\mathcal{Z} \to \RR$,
there is a function $g\in V$ such that $f(z) = g(z)$ for every $z\in
\mathcal{Z}$. It is an easy observation that an $N$-dimensional
$k$-interpolation space exists if and only if a $k$-regular map
$\RR^m\to \RR^N$ exists.

Clearly obtaining efficient upper bounds on the dimension of $V$ is important in this context.
This connection should be also a valuable source of inspiration for topologists and algebraic geometers, as other interpolation problems give rise to other classes of maps, reminiscent of $k$-regular ones. An important precedent here is Alexander-Hirschowitz Theorem dealing with interpolation together with derivatives.

Very recently lower bounds for the dimension $N$ of the ambient space of the embedding were further improved by
   Blagojevi\'c, Cohen, L\"uck and Ziegler for real \cite{Ziegler}
   and complex \cite{blagojevic_cohen_luck_ziegler_On_highly_regular_embeddings_2} affine spaces.
   
While over the last 50 years the methods of algebraic topology improved the lower bounds on $N$,
still very few examples of efficient $k$-regular maps are known.
Those that we know rely on specific constructions for small $k$
--- cf. \cite{interpolation__Shekhtman__poly}, \cite[Example~2.6]{Ziegler}   ---
or transversality arguments \cite{kreg6}, \cite{interpolation__Shekhtman__transversality} (see also Lemma~\ref{lem_projection_from_H_disjoint_with_secant} below).
Probably the reason why transversality arguments are not efficient here is
that they construct maps from any source manifolds, not just affine space.

All the presentation above can be repeated with the field of real numbers $\RR$ replaced with complex numbers $\CC$.
That is, we also look for $k$-regular continuous functions $\CC^m \to \CC^N$.

\subsection{Main results}
In the present paper we propose an approach relying on algebraic geometry.
The basic underlying idea is simple.
First, we consider a Veronese map (given by all monomials of fixed degree $d$).
Such maps, when the degree $d$ of the monomials is sufficiently high, are known to be $k$-regular.
Then, we project from a sufficiently high dimensional linear subspace $H$.
It turns out that the dimension of possible $H$ is closely related to the numerical properties
of the smoothable and Gorenstein loci of the punctual Hilbert scheme
(see below, and Section~\ref{sect_Hilbert} for more details about Hilbert scheme and its Gorenstein locus).

\begin{thm} \label{thm_main_bound_intro_any_k}
   There exist $k$-regular maps $\RR^m \to \RR^{(m+1)(k-1)}$ and $\CC^m \to \CC^{(m+1)(k-1)}$.
\end{thm}
For small values of $k$ or $m$, we can make these constructions stronger:
\begin{thm} \label{thm_main_bound_intro_small_k}
   If $k \le 9$ or $m\le 2$, then there exist $k$-regular maps $\RR^m \to \RR^{m(k-1)+1}$  and $\CC^m \to \CC^{m(k-1)+1}$.
\end{thm}

\renewcommand{\arraystretch}{1.3}
\def\ubound{\parbox{1.1cm}{\small upper\\bound}}
\def\lbound{\parbox{1.1cm}{\small lower\\bound}}
\def\exactvalue{\parbox{1.1cm}{\small exact\\value}}

\begin{table}[hbt]
\makebox[\textwidth][c]{
\begin{tabular}{ccc l c l cc l cc}
\toprule
$k$&\multicolumn{2}{c}{$m$ arbitrary}  &&$m = 1$&&\multicolumn{2}{c}{$m = 2$}&&\multicolumn{2}{c}{$m = 3$}\\
\cmidrule{2-3}\cmidrule{5-5}\cmidrule{7-8}\cmidrule{10-11}
&\lbound&\ubound& &\exactvalue&&\lbound&\ubound&&\lbound&\ubound\\
\midrule
\footnotesize{arbitrary}& ? &\footnotesize{$(m+1)(k-1)$} &&\cellcolor{JasnoZielony}$k$&&\footnotesize{$2k-\alpha_2(k)$}&\footnotesize{$2k-1$}&&\footnotesize{$3k- 2\alpha_3(k)$}&\footnotesize{$4k-4$}\\

\rowcolor{JasnoZielony}$2$&\multicolumn{2}{c}{$m+1$}&&$2$&&\multicolumn{2}{c}{$3$}&&\multicolumn{2}{c}{$4$}\\

\rowcolor{JasnoZielony}$3$&\multicolumn{2}{c}{$2m+1$}&&$3$&&\multicolumn{2}{c}{$5$}&&\multicolumn{2}{c}{$7$}\\

$4$&$2m+1$&$3m+1$&&\cellcolor{JasnoZielony}$4$&\cellcolor{JasnoZielony}&\multicolumn{2}{c}{\cellcolor{JasnoZielony}$7$}&& $8$  &$10$\\

\rowcolor{JasnoZielony}$5$&\multicolumn{2}{c}{$4m+1$}&&$5$&&\multicolumn{2}{c}{$9$}&&\multicolumn{2}{c}{$13$}\\

$6$& $4m+1$&$5m+1$&&\cellcolor{JasnoZielony}$6$&&$ 10 $&$11$&& $14$ &$16$\\

\rowcolor{JasnoZielony}$7$&\multicolumn{2}{c}{$6m+1$}&&$7$&&\multicolumn{2}{c}{$13$}&&\multicolumn{2}{c}{$19$}\\

$8$& $6m+1$&$7m+1$&&\cellcolor{JasnoZielony}$8$&\cellcolor{JasnoZielony}&\multicolumn{2}{c}{\cellcolor{JasnoZielony}$15$}&& 19 &$22$\\

$9$&$6m+1$&$8m+1$&&\cellcolor{JasnoZielony}$9$&&$ 16 $&$17$&&\multicolumn{2}{c}{\cellcolor{JasnoZielony} $25$ }\\

$10$&$6m+1$&$9m + 9$&&\cellcolor{JasnoZielony}$10$&&$ 18 $&$19$&& $26$ &$36$\\

{prime}&{$m(k-1)+1$}&{$(m+1)(k-1)$}&&\cellcolor{JasnoZielony}$k$&\cellcolor{JasnoZielony}&\multicolumn{2}{c}{\cellcolor{JasnoZielony}$2k-1$}&&{$3k-2$}&{$4k-4$}\\
\bottomrule
\end{tabular}}
\vspace*{\baselineskip}

\caption{Let $N_0 = N_0(k, m)$ be the minimal integer such that
there exists a $k$-regular map $\CC^m \to \CC^{N_0}$. The table provides
bounds
for $N_0$.  The upper bounds are obtained in this article. The lower bounds
have been obtained by
   \cite{blagojevic_cohen_luck_ziegler_On_highly_regular_embeddings_2}
   (see Theorem~\ref{thm_lower_bound}), or immediately follow from their result,
   since if there is no $(k-1)$-regular map then there is no $k$-regular map between the same spaces.
   Cases when the bounds coincide are shaded green.}\label{tab_complex_dimensions}
\end{table}

Particularly the dimensions arising in Theorem~\ref{thm_main_bound_intro_small_k} are very close to the lower bound 
obtained by Blagojevi\'c, Cohen, L{\"u}ck, and Ziegler by means of algebraic topology:   
\begin{thm}[{\cite{
 blagojevic_cohen_luck_ziegler_On_highly_regular_embeddings_2}}]\label{thm_lower_bound}
Let $k,m\in \n$ and let $p$ be a prime number. 
There is no $k$-regular map $\CC^m\rightarrow \CC^N$ if:
\begin{itemize}
\item $k=p$ and $N\leq m(k-1)$ or
\item $m=p^t$ for some $t\geq 1$ and $N\leq m(k-\alpha_p(k))+\alpha_p(k)-1$,
        where $\alpha_p(k)$ denotes the sum of coefficients in the $p$-adic expansion of $k$
        (i.e.~in the sum of digits in the presentation of $k$ in the numerical system with base $p$).
\end{itemize}
\end{thm}

\begin{rem}
  According to \cite{Ziegler} there should not be any $k$-regular map $\CC^m \to \CC^N$ 
    for $N \le m(k-\alpha_2(k))+\alpha_2(k)-1$ for any values of $k$ and $m$ 
    (without the assumption that $m$ is a power of $2$, as in Theorem~\ref{thm_lower_bound}). 
  However, from a private communication with the authors we know that the proof relies on a paper with a mistake.
  Hence, we decided not to refer to that result, 
    even though this bound is highly probable and the proof may soon be corrected and published as an erratum.
  As a consequence, we are not currently aware of any non-trivial lower bound in 
    Table ~\ref{tab_complex_dimensions} that works for general value of $k$ and $m$.
\end{rem}

Table~\ref{tab_complex_dimensions} compares the known lower bounds on $N$ with our results.
Theorems~\ref{thm_main_bound_intro_any_k} and \ref{thm_main_bound_intro_small_k} are proved in
   Subsection~\ref{sect_Veronese_embedding} and Section~\ref{sect_real_case}. In addition, Theorem~\ref{thm_main_bound_intro_small_k} relies on some statements proven in Appendix.
The dimension bounds in these theorems are related to dimensions of certain algebraic parameter spaces,
see~Theorems~\ref{thm_bound_by_hilbert_scheme_intro}, \ref{thm_bound_by_gorenstein_hilbert_scheme_intro}.
For small values of $k$, $m$, we can give explicit examples of such maps:

\begin{exm}\label{ex_3_regular_map}
   Consider the map $\CC^m\to \CC^{2m+1}$ given by
   \[
     (\fromto{t_1}{t_m}) \mapsto (1, t_1, t_1^2, t_2, t_2^2,\dotsc, t_m, t_m^2).
   \]
   This map is $3$-regular.
\end{exm}

Looking at Examples~\ref{exam_vandermonde} and \ref{ex_3_regular_map}, one may wonder if the map given by $i$-th powers of all variables for $i\in\setfromto{0,1}{k-1}$ is always $k$-regular.
\begin{exm}
   The map $\CC^2\to \CC^{7}$ given by
   \[
     (\fromto{s}{t}) \mapsto (1, s, t, s^2, t^2, s^3, t^3)
   \]
   is \textbf{not} $4$-regular. 
   The images of the points $(s_1, t_1), (s_1, t_2), (s_2, t_1), (s_2, t_2)$ are linearly dependent. 
\end{exm}
Instead, using a non-monomial map, one obtains:

\begin{exm}
   Consider the map $\CC^2\to \CC^{7}$ given by
   \[
      (s,t) \mapsto (1, s, t, s^2, st, t^2 - s^3, t^3).
   \]
   This map is $4$-regular.
\end{exm}

\begin{exm}
   Consider the map $\CC^3\to \CC^{10}$ given by
   \[
      (s,t,u) \mapsto (1, t,s,u, st, su, s^2-tu, t^2 -s^3, u^2 - t^3, u^3).
   \]
 This map is $4$-regular after restricting to a small open disc around $0 \in \CC^3$.
    Since the disc is homeomorphic with $\CC^3$, we obtain a continuous $4$-regular map $\CC^3\to \CC^{10}$.
\end{exm}
%

\begin{exm}
Consider the map $f_1 \colon \CC^2\to \CC^{9}$ given by
      \[
         (s,t) \mapsto (1, s, t, s^2, st, t^2, t^3, s^3-t^4, s^4).
      \]
   This map is $5$-regular.
\end{exm}

The following conjecture is motivated by the algebro-geometric approach we develop in this article.
\begin{con}
   There exists a continuous $k$-regular map $\CC^m\to \CC^N$ if and only if 
   \[
     {N \ge m(k-1) +1}.
   \]
\end{con}
As illustrated in Table~\ref{tab_complex_dimensions}, whenever the lower bound and upper bound on the minimal value of $N$ 
   coincide, both are equal to the conjectured value $m(k-1)+1$. 


\subsection{The Hilbert scheme}
The Hilbert scheme is an algebraic scheme, one of the simplest examples of moduli spaces.
We consider the Hilbert scheme that parametrizes all subschemes of a given projective space of a given length $k$.
It may have a complicated topological (reducible) and algebraic (non-reduced) structure.
However, it is always compact and connected \cite{Hartcon}.
One of its components, called the smoothable component,
  is a compactification of the configuration space parametrizing $k$ (distinct) points in the projective space.
Schemes corresponding to points of the smoothable component are called
\textbf{smoothable}.
Some points of the Hilbert scheme correspond to subschemes $S$ that are not reduced,
   thus are supported at less than $k$ points.
The locus of subschemes $S$ supported at precisely one point forms a closed subscheme
   of the Hilbert scheme called the \textbf{punctual Hilbert scheme}.
The punctual Hilbert scheme and schemes having an additional algebraic property of being Gorenstein
   \cite[Chapter~21]{Eisenbud}, \cite{iarrobino_kanev_book_Gorenstein_algebras} turn out to be of particular importance
   while studying $k$-regular maps.

\begin{thm}\label{thm_bound_by_hilbert_scheme_intro}
   Suppose $k$ and $m$ are positive integers,
      and let $d$ be the dimension of the locus of smoothable schemes in the punctual Hilbert scheme of length
      $k$ subschemes $\CC^m$.
   Then there exist $k$-regular maps $\RR^m \to \RR^{d+k}$ and $\CC^m \to \CC^{d+k}$.
\end{thm}

\begin{thm}\label{thm_bound_by_gorenstein_hilbert_scheme_intro}
   Suppose $m>0$ and $k \ge 2$ are two integers.
   Let $d_i$ be the dimension of the locus of Gorenstein schemes in the punctual Hilbert scheme of length
     $i$ subschemes of $\CC^m$.
   Then there exist $k$-regular maps $\RR^m \to \RR^N$ and $\CC^m \to \CC^N$,
      where $N = \max\set{d_i + i \mid 2\le i \le k}$.
\end{thm}
The latter bound is most relevant, whenever $k < k_0$  where $k_0 \in \ZZ \cup \{\infty\}$
  is the length of the shortest non-smoothable Gorenstein scheme in $\CC^m$.
Explicitly, $k_0 = \infty$ for $m = 1$, $2$, or $3$ (i.e.~all finite Gorenstein schemes in $\CC^m$ are smoothable for $m\le 3$),  $k_0 =14$ for $m \ge 6$, and $k_0$ is in between for $m = 4$, or $5$:
For $m=4$ it is known that $ 15 \le k_0 \le 140$ and for $m=5$ we have  $ 15 \le k_0 \le 42$.
In general, we can provide bounds on the required dimensions, and for small values of $k$, we can calculate them explicitly.
Both theorems are proven in Subsection~\ref{sect_Veronese_embedding} and Section~\ref{sect_real_case}.

Although we use methods of algebraic geometry, the maps we obtain in the end are not necessarily algebraic.
That is, under the assumptions of one of the theorems we construct a polynomial map from a small
  $m$-dimensional ball to $\RR^N$,
  or from a small $2m$ dimensional ball to $\CC^N$.
Then we use a (non-polynomial) homeomorphisms of the ball with $\RR^m$ or $\CC^m$.

The properties of punctual Hilbert scheme and the locus of Gorenstein
subschemes are intensively studied
   \cite{iarrobino_punctual_Hilbert_schemes, Briancon_surface, Granger} in algebraic geometry.
Thus our main theorem relates a very classical problem from topology with properties of a well-studied object in
  algebraic geometry, providing non-trivial estimates.
In order to prove our results, we need to analyse properties of the punctual Hilbert scheme.
We try to avoid the most technical methods at the price of proving Theorem~\ref{thm_main_bound_intro_small_k} only with $k\le 9$. 
With more lengthy and sophisticated methods one should be able to improve this bound to slightly higher values of $k$.
See Appendix for more details.

\subsection{Ideal of diagonal}
Although our principal motivations are drawn from topology,
   the following algebraic implications were pointed out to us by Bernd Sturmfels.
Suppose we have an algebraic map $\CC^m\rightarrow \CC^N$ given by $N$ polynomials $f_1,\dots,f_N$ in $m$ variables.
By taking the Cartesian product we obtain a map $(\CC^m)^k\rightarrow (\CC^N)^k$.
This defines a $k\times N$ matrix:
$$
M=
\left[
\begin{array}{cccc}
f_1(x^1)&f_2(x^1)&\dots&f_N(x^1)\\
f_1(x^2)&f_2(x^2)&\dots&f_N(x^2)\\
\vdots&\vdots&\vdots&\vdots\\
f_1(x^k)&f_2(x^k)&\dots&f_N(x^k)\\
\end{array}\right],
$$
where $x^i$ denotes the collection of $m$ distinct variables $x^i_1,\dots,x^i_m$.
Let $I$ be the ideal in the polynomial ring in variables $\fromto{x^1_1}{x^k_m}$,
   generated by $k\times k$ minors of $M$.
It is obvious that polynomials in $I$ always vanish on the big diagonal $D\subset (\FF^n)^k$
   (i.e.~the set of tuples where at least two points coincide),
   as for points in $D$ two rows of $M$ are equal.
The fact that the map $(f_1,\dots,f_N)$ is $k$-regular is equivalent to the fact that $I$ defines, as a set, the big diagonal.
In other words, by Hilbert's Nullstellensatz over an algebraically closed field, $(f_1,\dots,f_N)$ is $k$-regular if and only if $\rad I=I(D)$. The ideal of the big diagonal $I(D)$ is of interest in algebra cf.~\cite[Chapter 18.3]{miller}. Examples of $k$-regular maps can provide insight into equations defining $D$.

\subsection{Overview}

In Section~\ref{sect_notation} we introduce the notation and compare languages of topology and algebraic geometry.
In Section~\ref{sect_configuration_space} we describe the configuration space, its significance to $k$-regularity problem and its partial compactification.
The main character in Section~\ref{sect_Hilbert_scheme} is the Hilbert scheme of points: the algebraic analogue of a compactification of the configuration space.
In Section~\ref{sect_secants} we describe secant varieties and their variants, and we explain their role in the complex  versions  of the proofs of main results.
Finally, Section~\ref{sect_real_case} concludes with extending the results to the real case, and also to some other manifolds than $\RR^m$, $\CC^m$. 
In the Appendix we explain the technical part, namely the calculation of the dimensions of various loci of the Hilbert scheme.

\subsection*{Acknowledgements}

We are grateful to Pavle Blagojevi\'c, Anthony Iarrobino, and Bernd Sturmfels
  for helpful comments, interesting discussions, and suggesting references.
Buczy\'nski and Micha{\l}ek would like to thank Simons Institute for the Theory of Computing and the organisers of the thematic semester ``Algorithms and Complexity in Algebraic Geometry'' for providing an excellent environment for collaboration during the final phase of writing the article.
The article is written as a part of "Computational complexity, generalised Waring type problems and tensor decompositions",
a project within "Canaletto",
    the executive program for scientific and technological cooperation between Italy and Poland, 2013-2015.
The paper is also a part of the activities of AGATES research group.

\section{Notation}\label{sect_notation}

\subsection{Base field and the clash of cultures: topology vs algebraic geometry}

The article addresses a problem in topology via methods of algebraic geometry.
It is thus desirable that the content is accessible to experts in one of these fields of mathematics,
  who do not necessarily specialise in the other field.
However, there is some clash of notation and we try to explain it here.
As a general rule, we prefer the topological notation, which should be mostly known to both tribes.

Throughout the paper we work over the \textbf{base field} $\FF$, which is either the field of complex numbers $\CC$
   or the field of real numbers $\RR$.
The careful algebro-geometric reader will certainly notice that some of the lemmas can be generalised
   to the set-up of any other base-field perhaps with a little extra work.
To maintain the accessibility we restrict to $\RR$ or $\CC$.

For an integer $n$ by $\FF^n$ we mean the real or complex affine space in the sense of linear algebra, or analysis,
   but not the algebro-geometric scheme $\Spec \FF[\alpha_1,\dotsc, \alpha_n]$.
On $\FF^n$ we will mainly consider the standard Euclidean topology, that is the topology generated by open balls.
In particular, unless otherwise stated, continuous maps and open subsets are with respect to Euclidean topology.
We will also consider the Zariski topology, which consists of the complements of affine varieties,
  i.e.~of subsets defined by vanishing of a collection of polynomials in $\FF[\alpha_1,\dotsc, \alpha_n]$.
All uses of Zariski topology will have a clear reference, for instance, ``Zariski-open subset''.

Analogously, $\FF\PP^n$ is the \textbf{real or complex projective space}, and we follow analogous conventions.
More generally, an $\FF$-\textbf{manifold} (or, simply, a \textbf{manifold}, if the reference to the field is clear from the context) denotes a real $\ccC^{\infty}$-manifold (if $\FF=\RR$) or a complex holomorphic manifold (if $\FF=\CC$).
Typical cases of manifold $M$ we will consider are either smooth algebraic varieties over $\FF$,
  or (Euclidean) open subsets of such.
Even more specifically, $M$ will be mainly $\FF^n$, $\FF\PP^n$, or $\BB_{\FF}^n$,
   the $n$-dimensional \textbf{open ball} in $\FF^n$.
Note that here and throughout we always count the \textbf{dimension} with respect to the base field $\FF$.
That is, for instance, $\dim \BB_{\CC}^n=n$, since the ball has complex dimension $n$,
   even though the real dimension is $2n$.

Throughout the paper, the notion of \textbf{regularity} will refer to the $k$-regularity,
  as explained in the introduction and in Section~\ref{sect_independence}.
However, in algebraic geometry \emph{regular map} (or \emph{regular morphism})
  means a polynomial map defined everywhere,
  as opposed to a \textbf{rational map}, which is well defined only on some Zariski-open subset, or in other words,
   it is  allowed to have ``poles''.
Rational maps are denoted by $\dashrightarrow$, rather than the usual $\to$.
Also \emph{Castelnouvo-Mumford regularity} is an integer associated to a sheaf on a projective space,
  and various derivative notions also use the word ``regular''.
Here we will not use the Castelnouvo-Mumford regularity explicitly,
  and in order to avoid the confusion we will say \textbf{algebraic morphism}
  to mean a polynomial map between two algebraic varieties,
  which is defined everywhere.

\subsection{Notions of independence and regularity}\label{sect_independence}

Many different notions of independence can be considered: linear, affine, projective\dots
Each of them leads to a slightly different notion of $k$-regularity.
Although they are essentially equivalent, up to small modifications, all of them are used and applied,
so we include a brief discussion of these notions.
Note that for affinely regular maps, there is a shift in notation in \cite[Definition 2.4]{Ziegler}.

\dfi\label{def_reg}
Let $M$ be an $\FF$-manifold.
\begin{itemize}
 \item A continuous map $f\colon M \to \FF^N$ is \textbf{$k$-regular} (or \textbf{linearly $k$-regular}),
         if the image of any $k$ distinct points in $M$ is linearly independent,
         that is, linearly spans a $k$-dimensional linear space.
 \item A continuous map $f\colon M \to \FF^N$ is \textbf{affinely $k$-regular},
         if the image of any $k$ distinct points in $M$ is affinely independent,
         that is, affinely spans a $(k-1)$-dimensional affine space.
 \item A continuous map $f\colon M \to \FF\PP^N$ is \textbf{projectively $k$-regular},
         if the image of any $k$ distinct points in $M$ is projectively
         independent,
         that is, projectively spans a $(k-1)$-dimensional projective space.
\end{itemize}
\kdfi

\begin{exm}
   By analogy to Example~\ref{ex_3_regular_map} we illustrate the differences between the above types of regularity:
   \begin{itemize}
    \item The map $f \colon \FF^2 \to \FF^5$ given by
         $(t_1, t_2) \mapsto (1, t_1, t_1^2, t_2, t_2^2)$ is linearly $k$-regular.
    \item The map $f \colon \FF^2 \to \FF^4$ given by
         $(t_1, t_2) \mapsto (t_1, t_1^2, t_2, t_2^2)$ is (affinely) $k$-regular.
    \item The map $f \colon \FF^2 \to \FF\PP^4$ given by
         $(t_1, t_2) \mapsto [1, t_1, t_1^2, t_2, t_2^2]$ is projectively $k$-regular.
   \end{itemize}
\end{exm}

Essentially, we can go back and forth between these different types of regularities.

\begin{lema}\label{lem_going_back_and_forth_between_regularities}
   The following relations between regularities hold.
   \begin{enumerate}
    \item    Suppose $f\colon M \to \FF^N$ is affinely $k$-regular map.
             Consider the standard open embedding $\FF^N\hookrightarrow \FF\PP^N$ into the projective space.
             Then the composition $f'\colon M  \to \FF\PP^N$ is a projectively $k$-regular map.
    \item    Suppose $f\colon M \to \FF^N$ is linearly $k$-regular map. Then $0$ is not in the image of $f$.
             Consider the standard projectivisation map $\FF^N \setminus \set{0} \to  \FF\PP^{N-1}$
             onto the projective space.
             Then the composition $f' \colon M  \to \FF\PP ^{N-1}$ is a projectively $k$-regular map.
    \item    Suppose $f\colon M \to \FF\PP^N$ is projectively $k$-regular map.
             If, in addition, the image of $M$ avoids a hyperplane $\FF\PP^{N-1} \subset \FF\PP^N$,
                then $f'\colon M \to \FF^N = \FF\PP^N \setminus \FF\PP^{N-1}$ is affinely $k$-regular.
    \item    Suppose $f\colon M \to \FF^N$ is affinely $k$-regular map.
             Then the map $f'\colon M \to \FF^{N+1} = \FF \oplus \FF^N$ given by $f'(m) = (1,f(m))$ is linearly $k$-regular.
   \end{enumerate}
\end{lema}

Topologist are principally interested in the $k$-regular embeddings $M \hookrightarrow \RR^N$.
In algebraic geometry, the analogue of compact manifold is non-singular projective algebraic variety.
However, there is no non-constant algebraic morphism $M \to \FF^N$,
  if $M$ is a smooth projective algebraic variety.
Thus it is more natural to study the algebraic morphisms into projective space $\FF\PP^N$
  and ask, whether they are projectively $k$-regular.

On the other extreme, we may work locally.
If there exists a $k$-regular continuous map defined on an open subset $U\subset \FF^m$,
  then there exists also a $k$-regular map defined on the whole $\FF^m$.
This follows from the fact that an open ball in $\FF^m$ is homeomorphic (or diffeomorphic) to $\FF^m$.
Again, in the algebraic context this is not true.
There is a big difference between algebraic morphisms, i.e.~maps defined everywhere,
  and rational maps, i.e.~maps defined on a Zariski open subset.

\section{Configuration space in topology}\label{sect_configuration_space}

In this section we discuss the topological approach to $k$-regularity.
It involves the studies of configuration space of points on a manifold.
Next we illustrate, how sufficiently nice partial compactification of the configuration space
   may lead to interesting $k$-regular embeddings.
This construction contains the main idea of article.
In the following section we show how Hilbert scheme,
  an object from algebraic geometry, can be used to construct the partial compactification.

\subsection{Configuration space and \texorpdfstring{$k$}{k}-regularity}\label{sect_conf_space_and_k_regulatity}

Suppose $M$ is an $\FF$-manifold. The \textbf{(unordered) configuration space} $C_k(M)$
   is the set of unordered $k$-tuples of distinct points on $M$.
More precisely, it is the following manifold:
\begin{equation}\label{equ_define_configuration_space}
   C_k(M) = (\underbrace{M \times M \times \dotsb \times M}_{\text{$k$ copies}} \setminus \Delta) / \Sigma_k
\end{equation}
Here $\Delta$ is the big diagonal, i.e.~the set of all $k$-tuples of points in $M$, for which at least two points coincide,
  and $\Sigma_k$ is the permutation group on $k$ elements acting on the product by permuting factors.

Suppose $f\colon M \to \FF^N$ is a (linearly) $k$-regular map.
The $k$-regularity induces a map to the Stiefel manifold of $k$-frames $St(k, \FF^N)$, and hence to Grassmannian $Gr(k, \FF^N)$
   (i.e.~the parameter space for linear $k$-dimensional subspaces of $\FF^N$):
   \begin{alignat*}{2}
              \xi\colon C_k(M) & \to St(k, \FF^N) &&\to Gr(k, \FF^N)\\
              (\fromto{p_1}{p_k}) & \mapsto \left(\fromto{f(p_1)}{f(p_k)}\right)&&\mapsto \left\langle \fromto{f(p_1)}{f(p_k)} \right\rangle,
   \end{alignat*}
where $ \langle \dotsc \rangle$ denotes the linear span in $\FF^N$.
Thus there is the induced vector bundle $\ccE$ on $C_k(M)$,
  which is the pull-back of the universal subbundle on the Grassmannian.
Explicitly, $\ccE$ is the incidence set $\ccE \subset C_k(M) \times \FF^N$:
\[
   \ccE = \set{(\fromto{p_1}{p_k}), v
                         \mid  (\fromto{p_1}{p_k}) \in C_k(M), v \in \langle \fromto{f(p_1)}{f(p_k)} \rangle \subset \FF^N}.
\]
It admits natural projections into $\FF^N$ and onto $C_k(M)$.
We let the \textbf{open secant variety $\sigma_k^{\circ}(f(M))$} be the image of $\ccE$ in $\FF^N$.

Analogously, if $f$ is an affinely (or projectively) $k$-regular map, we can define the map
   $\xi \colon C_k(M) \to Gr^{aff}(k-1, \FF^N)$ (or $\xi \colon C_k(M) \to Gr(\FF\PP^{k-1}, \FF\PP^N)$),
   the affine space bundle (or projective space bundle) $\ccE$,
   and the affine or projective analogue of the open secant variety $\sigma_k^{\circ}(f(M))$.
We only state Lemma~\ref{lem_projection_from_H_disjoint_with_secant}
   for the affinely regular case.
The reader will easily generalise the lemma to the other cases.

\begin{lema}\label{lem_projection_from_H_disjoint_with_secant}
  Suppose $f\colon M \to \FF^N$ is an affinely $k$-regular map
     and $H \subset \FF^N$  is an affine subspace which is disjoint from $\sigma_k^{\circ}(f(M))$.
  Then there exists an affinely $k$-regular map $f'\colon M \to \FF^{N-\dim H-1}$.
\end{lema}

\begin{proof}
  Pick any affine subspace $\FF^{N-\dim H -1} \subset \FF^N$ disjoint from $H$.
  Let $\pi\colon \FF^N \to \FF^{N-\dim H-1}$ be the projection with centre $H$,
    and set $f' = \pi \circ f$ to be the composition.
  For any $k$ points $(\fromto{p_1}{p_k})$ on $M$, their image under $f$ spans an affine space
    $\FF^{k-1} \subset \sigma_k^{\circ}(f(M))$, which is disjoint from $H$.
  Thus the image $\pi (\FF^{k-1})$ has dimension $k-1$ and $f'$ is $k$-regular.
\end{proof}

A standard way to use the lemma is to bound $\dim \sigma_k^{\circ}(f(M)) \le k \dim M + k -1$
  and use the generic transversality.
This argument shows that any manifold of dimension $m$ admits a $k$-regular embedding into
  $k m + k -1$. See also \cite{kreg6}.

In what follows we will find a bit better way to avoid the secant variety,
  at the price of restricting to $M=\BB_{\FF}^m \simeq \FF^m$.

\subsection{The main idea}

We commence with an informal discussion of the main idea.
Suppose we have a $k$-regular map $f \colon\BB_{\FF}^m \to \FF^N$ 
                  (perhaps $N$ is very large in the first place).
We are looking for a criterion that will allow us to reduce $N$.
By Lemma~\ref{lem_projection_from_H_disjoint_with_secant} this is possible, 
  if the union of secants $\langle \fromto{f(p_1)}{f(p_k)}\rangle$  does not fill the entire $\FF^N$.
Pick a point $x \in \FF^N$. 
Suppose on the contrary, that the secants fill the $\FF^N$, in particular, $x \in \langle \fromto{f(p^1_1)}{f(p^1_k)}\rangle$.
Then we may reduce the disc, say by halving the radius.
Now, if $x$ is not in any $k$-secant, we are done, and we can reduce $N$.
So suppose $x \in \langle \fromto{f(p^2_1)}{f(p^2_k)}\rangle$, 
   but now the points $p^2_i$ are much closer to each other. 
We keep halving the radius of the disc.
In the end, either $x$ stops being in a secant, 
or it is contained in the \emph{limit} of the secants (perhaps after taking a subsequence).
Here the limit of points is the centre of the discs.
The principal idea is that there is much less of such limits, than of all secants.

To formalise the above idea we introduce the notion of partial compactification in the following subsection.
   
\subsection{A partial compactification} \label{sect_partial_compactification}

We now switch to the projective set-up, in order to  benefit from compactness.
Given an affinely or linearly $k$-regular map, we
  replace it with a projectively $k$-linear map $f \colon M \to \FF\PP^N$
  using Lemma~\ref{lem_going_back_and_forth_between_regularities}.
By analogy to Subsection~\ref{sect_conf_space_and_k_regulatity} consider the map
  $\xi\colon C_k(M)  \to Gr(\FF\PP^{k-1}, \FF\PP^N)$ into the (compact) Grassmannian,
  and the projective space bundle $\ccE$ on $C_k(M)$,
  which is the pull-back of the universal projective space subbundle on the Grassmannian.
In particular, the bundle map $\ccE \to C_k(M)$
  is \textbf{proper}, that is the preimage of every compact subset is compact.
We also let $\sigma_k^{\circ}(f(M))$ be the image of $\ccE$ in $\FF\PP^N$.

The idea to improve the construction of $k$-regular embeddings into small ambient spaces is following.
Suppose we can find a partial compactification $\overline{C}_k(M)$ of $C_k(M)$ satisfying the following conditions:
\renewcommand{\theenumi}{(\alph{enumi})}
\renewcommand{\labelenumi}{\theenumi}
\begin{enumerate}
  \item \label{item_partial_compactification_density}
        $C_k(M) \subset \overline{C}_k(M)$ and $C_k(M)$ is dense in $\overline{C}_k(M)$.
  \item \label{item_partial_compactification_properness}
        The natural embedding map $C_k(M) \to \Sym^k(M)$ extends
          to a proper map
          \[
             \rho \colon \overline{C}_k(M) \to \Sym^k(M);
          \]
        here $\Sym^k(M) = M^{\times k} / \Sigma_k$ is the symmetric product of $M$.
  \item \label{item_partial_compactification_map_to_Grassmannian}
        The map $\xi\colon C_k(M)  \to Gr(\FF\PP^{k-1}, \FF\PP^N)$ extends to a map
                \[
                  \overline{\xi}\colon \overline{C}_k(M)  \to Gr(\FF\PP^{k-1}, \FF\PP^N).
                \]
\end{enumerate}

Note that $\overline{C}_k(M)$ is compact if and only if $M$ is compact.
Thus $\overline{C}_k(M)$ fills in the ``holes'' in  $C_k(M)$ which are the results of removing the big diagonal $\Delta$
   in \eqref{equ_define_configuration_space}, but does not compactify $M$.
This is why we refer to $\overline{C}_k(M)$ as \textbf{a partial compactification}.

We will construct $\overline{C}_k(M)$ using a Hilbert scheme in
Section~\ref{sect_secants}.
Here we show how the partial compactification is useful to construct affine spaces disjoint from $\sigma_k^{\circ}(f(M))$,
  perhaps after shrinking $M$.

Pick any point $p \in M$ and its small open neighbourhood $\BB_{\FF}^m$ in $M$.
Let $p^k \in \Sym^k (M)$ denote the $k$-tuple consisting of $k$ copies of $p$.
Consider the preimages via $\rho \colon \overline{C}_k(M) \to \Sym^k(M)$
   of $p^k$ and $\Sym^k (\BB_{\FF}^m)$.
Informally, $\Sym^k (\BB_{\FF}^m)$ is a small neighbourhood of $p^k$ which contains only configurations of points with repetitions, that are within at most fixed small distance to $p$. 
   
Note that $\rho^{-1}(\Sym^k (\BB_{\FF}^m))$ is an open neighbourhood of the compact set $\rho^{-1}(p^k)$.
Let $\overline{\ccE}$ be the pull-back of the universal subbundle from the Grassmannian to $\overline{C}_k(M)$.
Further denote $\ccE_{p}$ and $\ccE_{\BB_{\FF}^m}$ to be the restrictions of $\overline{\ccE}$
   to $\rho^{-1}(p^k)$ and $\rho^{-1}(\Sym^k(\BB_{\FF}^m))$, respectively.
Then we can look at the images of $\ccE_{\BB_{\FF}^m}$ and $\ccE_{p}$ in $\FF\PP^N$.
On one hand, the image of $\ccE_{\BB_{\FF}^m}$ contains $\sigma_k^{\circ}(f(\BB_{\FF}^m))$.
On the other hand, in Lemma~\ref{lem_projection_from_H_disjoint_with_areole}
  we observe, that if a linear subspace $H \subset \FF\PP^N$
  avoids the image of $\ccE_{p}$,
  then it also avoids the image of $\ccE_{\BB_{\FF}^m}$, perhaps after further shrinking $\BB_{\FF}^m$.

\begin{lema}\label{lem_projection_from_H_disjoint_with_areole}
   Suppose $f\colon M \to \FF\PP^N$ is $k$-regular, and $\overline{C}_k(M)$ satisfies the conditions
   \ref{item_partial_compactification_density}--\ref{item_partial_compactification_map_to_Grassmannian} above.	
   If a (projective) linear subspace $H$ avoids the image $\gota_p$ of $\ccE_{p}$
     under the map $\overline{\ccE} \to \FF\PP^N$,
     then there exists an affinely $k$-regular map $g \colon \FF^m \to \FF^{N- \dim H -1}$.

\end{lema}

\begin{proof}
    First observe that $\ccE_{p}$ is the preimage of a compact set $\rho^{-1}(p^k)$
  under a proper bundle map $\overline{\ccE} \to \overline{C}_k(M)$, thus it
  is compact. Since  $\gota_p$ is the image of $\ccE_{p}$, it is compact as
  well.

Let $H \subset \FF\PP^N$ be the projective linear subspace avoiding $\gota_p$.
It is also compact, hence avoids an open (tubular) neighbourhood of $\gota_p$.
Hence $H$ avoids the image of $\ccE_{\BB_{\FF}^m}$, perhaps after shrinking the ball $\BB_{\FF}^m$.
Thus $H$ avoids $\sigma_k^{\circ}(f(\BB_{\FF}^m))$ and we are in position to apply the projective version of
  Lemma~\ref{lem_projection_from_H_disjoint_with_secant}.
In particular, there exists a projectively $k$-regular map $f'\colon \BB_{\FF}^m \to \FF\PP^{N-\dim H -1}$.
Perhaps after further shrinking $\BB_{\FF}^m$, we may assume the image of $f'$ avoids a hyperplane
$\FF\PP^{N-\dim H -2} \subset \FF\PP^{N-\dim H -1}$.
Thus by Lemma~\ref{lem_going_back_and_forth_between_regularities},
  there exists an affinely $k$-regular map  $g\colon  \FF^m \simeq \BB_{\FF}^m \to \FF^{N-\dim H -1}$.
\end{proof}

If we can assure $\gota_p$ has some sufficiently good structure, for example it is a manifold, or an algebraic variety,
   then we can again use transversality type arguments to show that there exist $k$-regular maps
   into the space of dimension equal to $\dim \gota_p$.

\section{Hilbert scheme and finite Gorenstein schemes}\label{sect_Hilbert_scheme}

 \subsection{Finite subschemes of \texorpdfstring{$\FF^n$}{Fn}.}
Let us briefly explain what is the intuition behind
a finite subscheme $R$ of an affine space $\FF^n$.
As usual in algebraic geometry, 
   we analyse $R$ by looking at restrictions of functions (polynomials) from
$\FF^n$ to $R$.
If $R = \{p_1, \ldots , p_k\}$ is a set of points then a function $f$ restricted to $R$ is
just the set of values $f(p_1), \ldots , f(p_k)$.  For general $R$, the restriction of $f$ may contain
additional higher-order data of $f$ near $p_i$, e.g.~its partial derivatives
at $p_i$, second-order derivatives at $p_i$ etc.

In fact, finite subschemes of $\FF^n$ are defined using the above
intuition.
Denote by $S$ the ring of algebraic functions  (i.e.~polynomials) on $\FF^n$.
Then the ring of algebraic functions on $R$ is
$S/(\mathfrak{q}_1\cap \mathfrak{q}_2\cap  \ldots \cap \mathfrak{q}_k)$,
  where $\mathfrak{q}_i$ is an ideal of finite colength, whose radical is the maximal ideal of the functions vanishing at point $p_i$.
We say that $\mathfrak{q}_1\cap \mathfrak{q}_2\cap  \ldots \cap \mathfrak{q}_k$
  is the \textbf{ideal of functions vanishing on $R$}.
Such scheme $R$ is denoted
\[
   \Spec S/(\mathfrak{q}_1\cap \mathfrak{q}_2\cap  \ldots \cap \mathfrak{q}_k).
\]
The set of points $p_1, \ldots ,p_k$ is called the \textbf{support} of $R$.
The number $\mu_{p_i} = \mu_{p_i}(R) := 
\dim_\FF(S/\mathfrak{q}_i)$
is called the
\textbf{multiplicity} of $R$ at $p_i$. It measures the
complexity of $R$ near~$p_i$. The \textbf{length} of $R$ is $\sum \mu_{p_i}$.

\begin{exm}
Let $R = \{p_1, \ldots , p_k\}$ be a set of points. Then the
support of $R$ is $\{p_1, \ldots , p_k\}$ and the multiplicity of $R$ at each
$p_i$ equals $1$, thus the length of $R$ is $k$.
\end{exm}

\begin{exm}\label{exam_arrow_and_point}
Let $\FF^2$ have coordinates $x, y$, then $S = \FF[x, y]$. Let
$R=\Spec \FF[\alpha, \beta]/((\alpha^2, \beta)\cap (\alpha-1, \beta-2))$ be a finite subscheme of $\FF^2$ with
support $p_1 = (0, 0), p_2 = (1, 2)$ and multiplicities $2$ at $p_1$ and
$1$ at
$p_2$. For any $f\in S$ the restriction of $f$ to $R$ consists
of the value $f(p_1)$ together with the derivative $\partial_{\alpha}(f)(p_1)$ and
the value $f(p_2)$.
\end{exm}

We may define finite subschemes of any affine variety $\mathcal{A}$ by
replacing $\FF^n$ by $\mathcal{A}$ in the above discussion.
Any projective manifold $X \subseteq \mathbb{P}^n$ is covered by affine
varieties.  A finite subscheme of $X$ is just a finite subscheme of any of
those affine varieties.

\subsection{Introduction to Hilbert scheme of points}\label{sect_Hilbert}

Let $X\subseteq \FF\PP^n$ be a smooth projective manifold. A reader may find it
convenient to assume that $X = \FF\PP^n$.
Good references for the following discussion of Hilbert scheme of points are
e.g.~\cite{Gottsche_Hilbert_schemes_and_Betti_numbers,
Stromme_Intro_to_Hilbert} and \cite[Chapters 5 and 7]{fantechi_et_al_fundamental_ag}.

\begin{df}
The \textbf{Hilbert scheme of $k$ points} of $X$ is the set of finite
subschemes of $X$ with length $k$. It is denoted by $\Hilb_k(X)$. There is
a unique natural scheme structure on $\Hilb_k(X)$
making it a compact projective scheme.
\end{df}
The construction of the Hilbert scheme is due to Grothendieck: we take $d \gg
0$,
a Veronese reembedding $\nu_d: \FF\PP^n\to \FF\PP^N$ and to any
subscheme $R \subseteq X \subseteq \FF\PP^n$ assign the projective span
$\langle \nu_d(R) \rangle$, see Definition~\ref{df_linearspan}. This gives a map to the Grassmannian of
$\FF\PP^N$, which turns out to be Zariski-closed embedding, see
\cite[Chapter 5]{fantechi_et_al_fundamental_ag}. In particular, $\Hilb_k(X)$ is compact.

Algebraic sets (or schemes) may consist of several ``pieces'', called Zariski-irreducible components.
A single Zariski-irreducible component is also called a variety, and it is defined in terms of Zariski topology.
A set $X$ is Zariski-reducible if and only if it is a union of two Zariski-closed subsets $X=Y_1\cup Y_2$, with $Y_1\nsubseteq Y_2$ and $Y_2\nsubseteq Y_1$. It is Zariski-irreducible if it is not Zariski-reducible.
For example, a line or plane is Zariski-irreducible. A union of three axes in a three dimensional affine space is Zariski-reducible, each of the axes is a Zariski-irreducible component.
The Hilbert scheme and its special loci tend to be Zariski-reducible.

Every set of $k$ pairwise distinct points of $X$ is a finite
subscheme $R \in \Hilb_k(X)$, thus we have a natural map $C_k(X) \to
\Hilb_k(X)$. This map is a Zariski-open embedding. The Zariski-closure
$\overline{C}_k(X) \subseteq \Hilb_k(X)$ is a Zariski-irreducible component of
$\Hilb_k(X)$, which we call the \textbf{smoothable component} and denote
$\HilbSm_k(X)$. A finite subscheme $R \in \Hilb_k(X)$ is called
\textbf{smoothable} if it
lies in $\HilbSm_k(X)$.
Intuitively, smoothable component is the tame part of $\Hilb_k(X)$, whereas
other components are wilder.
If $X$ is a curve or a surface, then $\Hilb_k(X) = \HilbSm_k(X)$, but this
equality does not hold when $\dim X > 2$ and $k\gg 0$.


For the purposes of Section~\ref{sect_secants} we introduce finite Gorenstein
schemes.
We say that a finite subscheme $R$ of length $k$ is \textbf{Gorenstein} if
there are only finitely many subschemes $R' \subseteq R$ of length $k-1$.
See e.g.~\cite[Chapter 21]{Eisenbud} for a more standard definition,
see~\cite[Lemma~2.3]{nisiabu_jabu_cactus}
for the equivalence of our definition with the standard one, see Subsection~\ref{sect_apolarity} in Appendix for another equivalent description.
\begin{exm}
    If $R = \{ p_1, \ldots , p_k\}$ is a set of points, then the only
    subschemes of length $k-1$ are subsets of $k-1$ points of $R$. There is
    finitely many of such, thus $R$ is Gorenstein.
\end{exm}

\begin{exm}
    Suppose $R = \Spec \FF [\alpha,\beta]/ ((\alpha^2,\beta) \cap (\alpha-1, \beta-2))$ is as in Example~\ref{exam_arrow_and_point}. Then $\length R =3$ and the only subschemes of length $2$ 
       are
       \[
         \Spec \FF [\alpha,\beta]/ ((\alpha^2,\beta)) \text{ and } \Spec \FF [\alpha,\beta]/ ((\alpha,\beta) \cap (\alpha-1, \beta-2)).
       \]
 Thus $R$ is Gorenstein.  
\end{exm}

\begin{exm}
    If $R = \Spec \FF[\alpha, \beta]/(\alpha^2, \alpha\beta, \beta^2)$,
    then the length of $R$ is equal to $3$. For every $c \in \FF$ the scheme
    $\Spec \FF[\alpha, \beta]/\left((\alpha^2, \alpha\beta, \beta^2) + (\alpha - c\beta)\right)$
      is a subscheme of $R$ of length~$2$.
    There are infinitely many of those, thus $R$ is not Gorenstein.
\end{exm}

Denote by $\HilbGor_k(X) \subseteq \Hilb_k(X)$ the subset of finite Gorenstein
subschemes of length $k$. This is a \textbf{Zariski-open subset}, but it is usually not
Zariski-dense, i.e.~not all Zariski-irreducible components of $\Hilb_k(X)$
intersect $\HilbGor_k(X)$.
Moreover $C_k(X) \subseteq \HilbGor_k(X)$. Thus $\HilbGor_k(X)$ may be thought
as a link between $C_k(X)$ and $\Hilb_k(X)$, which forgets about some of the
wild components.

If $\dim X \leq 3$, then $\HilbGor_k(X) \subseteq \HilbSm_k(X)$. For $\dim X \geq 4$ and $k\gg 0$
this containment does not hold. That is, $\HilbGor_k(X)$ is not necessarily  Zariski-irreducible.
As summary we have the following open inclusions and closed irreducible components (also illustrated in the top part of Figure~\ref{fig_Hilb_k}):
\[
  \begin{array}{rcl}
    \Hilb_k(X) &\stackrel{\text{irr.~comp.}}{\supset}& \HilbSm_k(X) \\
     \cup \text{ {\tiny open}} & & \ \cup \text{ {\tiny open, Zariski-dense}} \\
    \HilbGor_k(X) &\stackrel{\text{open}}{\supset}& C_k(X)
    \end{array}
\]

\begin{figure}[htb]
\centering
\includegraphics[width=0.8\textwidth]{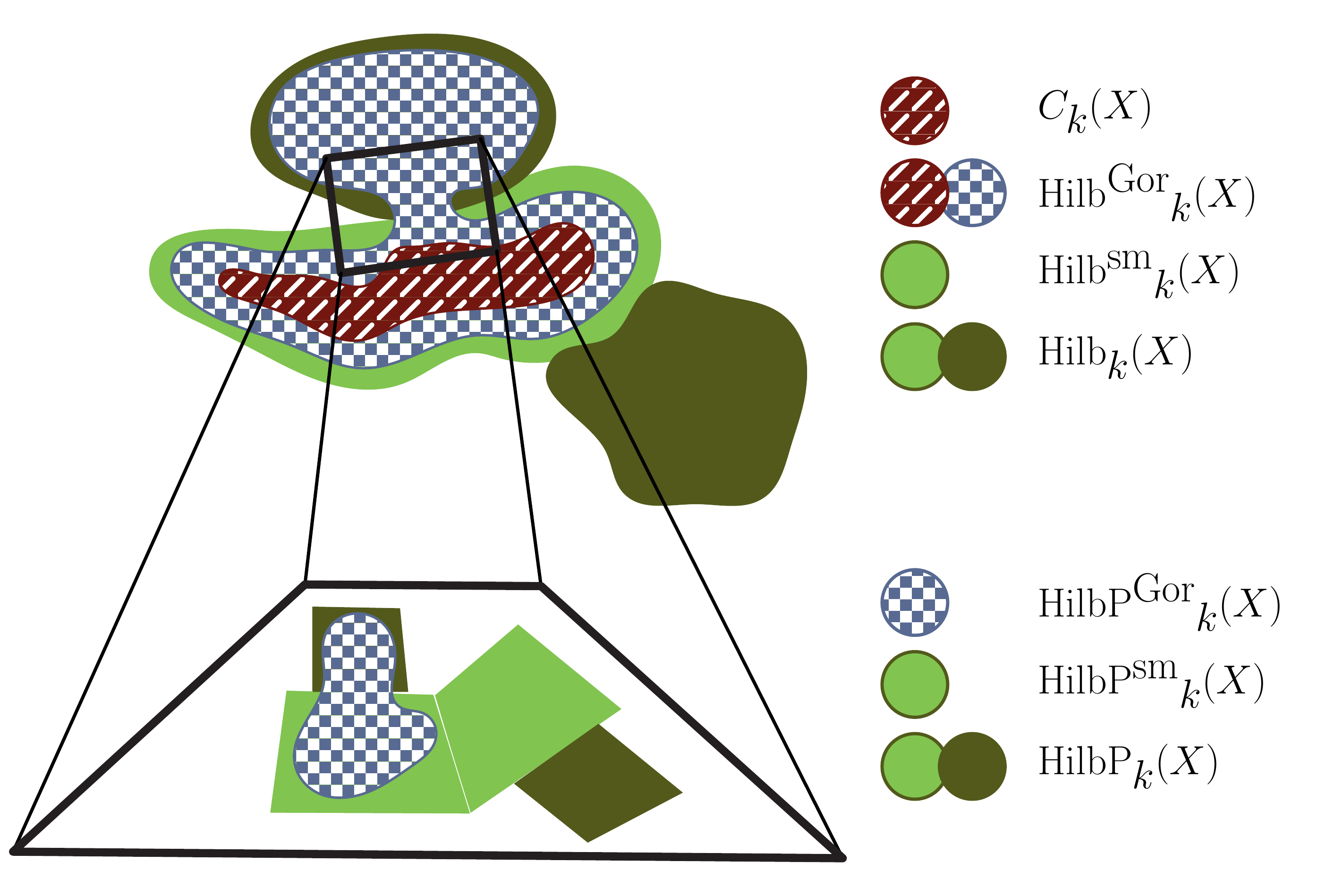}
\caption{Above: The Hilbert scheme $\Hilb_k(X)$ and relations between smoothable locus $\HilbSm_k(X)$, Gorenstein locus $\HilbGor_k(X)$, and the configuration space $C_k(X)$.\newline
Below: The punctual Hilbert scheme $\HilbP_k(X)$ - a closed subscheme of $\Hilb_k(X)$ - and relations between its smoothable locus $\HilbPSm_k(S)$ and Gorenstein locus $\HilbPGor_k(X)$.}
\label{fig_Hilb_k}
\end{figure}

\subsection{Punctual Hilbert scheme}\label{sect_punctual}

\def\HC{\rho}%
Let $X$ be a projective manifold.
Recall that $\Sym^k(X) = X^{\times k}/\Sigma_k$ is the symmetric product of
$k$ copies of $X$.
We have a natural map, called the \textbf{Hilbert-Chow morphism}
\[
  \HC:\Hilb_k(X) \to\Sym^k(X)\quad \text{such that} \quad \HC(R) = \bigcup_{p}
     p^{\mu_p(R)}
\]
   where the union is taken over the (finite) support of $R$,
   and $p^{\mu_p(R)}$ is the point $p$ counted with the multiplicity $\mu_p(R)$.
This map is proper.

\begin{df}
    Let $p\in X$.
    The \textbf{punctual Hilbert scheme} $\HilbP_k(X, p)$ is the set of finite schemes of
    $\Hilb_k(X)$ which are supported at a single point $p\in X$, i.e.~$\HilbP_k(X, p) = \rho^{-1}(p^k)$.

    The \textbf{Gorenstein}
    (respectively, \textbf{smoothable}) punctual Hilbert scheme is the set of Gorenstein (respectively, 
    smoothable) finite schemes in $\HilbP_k(X, p)$, denoted by
    $\HilbPGor_k(X, p)$ (respectively,  $\HilbPSm_k(X, p)$).
\end{df}

Note that since $\HilbP_k(X, p)$ is the preimage under $\HC$ of the point $p^k \in \Sym^k (X)$, it is Zariski-closed in $\Hilb_k(X)$ and proper.
Therefore, also $\HilbPSm_k(X, p)$ is proper and the corresponding Gorenstein subsets
are Zariski-open.
We stress, that despite $\HilbSm_k(X)$ is Zariski-irreducible by its definition,
   the intersection with the punctual Hilbert scheme $\HilbPSm_k(X, p)$ may be Zariski-reducible.
Summarising:
\[
  \begin{array}{rcl}
    \HilbP_k(X,p) &\stackrel{\text{union of irr.~comp.}}{\supset}& \HilbPSm_k(X,p) \\
     \cup \text{ {\tiny open}} & &  \\
    \HilbPGor_k(X,p)
 
  \end{array}
\]
(see also the bottom part of Figure~\ref{fig_Hilb_k}).

For the remainder of this section suppose that $X$ is a projective manifold of dimension
$m$ over $\FF = \CC$. Let $p\in X$ and $p_0\in \FF\PP^m$ be any points. Then there exist
analytically isomorphic neighbourhoods of $p$ in $X$ and $p_0$ in $\FF\PP^m$.
One can prove that $\HilbP_k(X, p)$ depends only on a Euclidean 
neighbourhood of $p$ in $X$. Therefore
\[
   \HilbP_k(X, p) \simeq \HilbP_k(\FF\PP^m, p_0)
\]
and we get analogous isomorphisms for Gorenstein and smoothable
punctual Hilbert schemes. In particular, when analysing the properties of
$\HilbP_k(X, p)$ for smooth $X$ we may always assume that $X$ is a projective
space. Moreover we obtain $\HilbP_k(X, p)  \simeq \HilbP_k(X, q)$ for 
any points $p, q\in X$.

\section{Secant varieties and areole in the complex
setting}\label{sect_secants}

As we have seen the Hilbert scheme serves as an analogue of the configuration space in topology and it naturally comes
 as a projective (thus compact) variety. We commence with a notational convention.

 \dfi[linear span $\langle\cdot \rangle$]\label{df_linearspan}
For any scheme $R$ contained in a projective space we denote by $\langle R \rangle$ 
   the smallest linearly embedded projective subspace containing $R$. We say that $R$ spans $\langle R \rangle$.
The space $\langle R \rangle$ is defined by all linear  equations that vanish when restricted to  $R$
   and is called the \textbf{linear span} of $R$.
\kdfi

Note that we always have $\dim \langle R \rangle \le \length R -1$. For instance, $k$ distinct points may span a projective space of dimension at most $k-1$.

\subsection{Areoles}

For this subsection we restrict to the algebraically closed base field $\FF= \CC$.

\begin{propdef}
   Suppose $X \subset \PP V
$ is a projective variety.
   The \textbf{$k$-th secant variety of $X$} is
   \begin{align*}
     \sigma_k(X)&:= \overline{\bigcup_{\fromto{x_1}{x_k}\in X}  \langle \fromto{x_1}{x_k}\rangle}\subset \PP V=\\
                & = \overline{ \bigcup \set { \langle R \rangle \mid   {R\in \Hilb_{\leq k}(X)}, R\text{ is smoothable in }X } }.
   \end{align*}
\end{propdef}
The equivalence of the two definitions is clear: Smoothable schemes are the limits of smooth schemes (disjoint reduced points),
  hence the linear spans of smoothable schemes are contained in the limits of spans of smooth schemes.

\textbf{Cactus varieties} are introduced in \cite{nisiabu_jabu_cactus} as a generalisation of secant varieties.
The stem of the cactus represents the secant variety, while the spines represent
   linear spans of non-smoothable Gorenstein schemes, typically supported at a single point.
Here we work with another set, \textbf{areole}, obtained in a similar way.
It consist of linear spans of only the smoothable schemes supported at a single fixed point.
It is therefore contained in the intersection of stem and spines.
In botany, \emph{areoles} are the parts of the cactus stems out of which the spines grow.
Remarkably, also beautiful cactus flowers grow out of areole,
  and analogously, regularity theorems spring out of the properties of the areole variety.

\begin{df}
   Suppose $X \subset \PP V$ is a projective variety and $p\in X$ is a point.
   The \textbf{$k$-th areole} at $p$ is
   \begin{align*}
     \gota_k(X,p)&:= \overline{ \bigcup \set { \langle R \rangle \mid   {R\in \HilbP_{\leq k}(X, p)},
                              R\text{ is smoothable in }X } }.
   \end{align*}
\end{df}
Note that we necessarily have $\gota_k(X,p) \subset \sigma_k(X)$.
Considering also shorter schemes, i.e.~$\Hilb_{\leq k}$ and $\HilbP_{\leq k}$
   is convenient for technical reasons,
   but essentially equally well we may consider $\Hilb_{k}$ and $\HilbP_{k}$,
   see \cite[Remark~2.4]{nisiabu_jabu_cactus} for more details.

\ex
If $p$ is a smooth point of $X$, then $\gota_2(X,p)$ is the (embedded, projective) tangent space to $X$ at $p$.
More generally, when $X$ is not smooth at $p$,
then $\gota_2(X,p)$ is called the tangent star.
It always contains the tangent cone.
\kex

We underline, that in general, in all the definitions of secant variety and areole we have to consider the closure:
  the union of linear spans is not necessarily closed,
  and it does not necessarily have the structure of algebraic variety.
However, in order for the smoothable Hilbert scheme to have the properties
\ref{item_partial_compactification_density}--\ref{item_partial_compactification_map_to_Grassmannian}
in Section~\ref{sect_partial_compactification}, we should better get rid of the closure.
For this purpose we define:
\dfi\label{def_strongly_regular}
Let $X$ be a projective variety over $\CC$, and suppose $f\colon X \to \CC\PP^N$
  is an algebraic morphism.
We say $f$ is \textbf{strongly projectively $k$-regular},
         if the image of any smoothable length $k$ zero dimensional subscheme of $X$
         spans a projective space of dimension $k-1$, i.e.~$\dim \langle R\rangle$ is maximal possible for every scheme $R \in \HilbSm_k(X)$.
\kdfi

\begin{lema}\label{lem_map_to_grassmannian_and_no_closure}
   Suppose $f\colon X \to \CC\PP^N$ is strongly projectively $k$-regular.
   Then $f$ determines a natural algebraic morphism:
   \[
      \HilbSm_{k} X \to Gr(\CC\PP^{k-1}, \CC\PP^{N}),
   \]
   and the closures in the definitions of secant varieties and areole are redundant:
   \begin{align*}
     \sigma_k(X)& =  \bigcup \set { \langle R \rangle \mid   {R\in \HilbSm_{\leq k}(X)}},\\
     \gota_k(X,p)& =  \bigcup \set { \langle R \rangle \mid   {R\in \HilbPSm_{\leq k}(X, p)}} .
   \end{align*}
\end{lema}

For a proof see \cite[Proposition~2.5]{nisiabu_jabu_cactus}, \cite[Lemma~2.6]{jabu_ginensky_landsberg_Eisenbuds_conjecture}, or \cite{bernardi_gimigliano_ida}.

It follows that in this case $\overline{C}_k(X) = \HilbSm_{k} X$ satisfies the conditions \ref{item_partial_compactification_density}--\ref{item_partial_compactification_map_to_Grassmannian}
  of compactification of the configuration space:
  \ref{item_partial_compactification_density} holds by the definition of smoothable schemes, \ref{item_partial_compactification_properness} holds since the Hilbert-Chow morphism is proper, and
  \ref{item_partial_compactification_map_to_Grassmannian} holds by Lemma~\ref{lem_map_to_grassmannian_and_no_closure}.

\begin{thm}\label{thm_k_regular_to_space_of_dim_same_as_areole}
   Suppose $\FF=\CC$, $X$ is a projective algebraic variety of dimension $m$,
      and $f \colon X \to \CC\PP^N$ is a strongly projectively $k$-regular algebraic morphism.
   Fix a smooth point $p \in f(X)$.
   Then there exists an affinely $k$-regular continuous map $\CC^m \to \CC^{N'}$, where $N'=\dim \gota_k(f(X), p)$.
\end{thm}
Note that the claim of the theorem does not depend on the initial value of $N$,
   thus it is enough to use a $k$-regular embedding with very high value $N$.
In the rest of this subsection we prove the theorem.
In Subsection~\ref{sect_areole_and_punctual_Hilb}
  we relate the bounds on dimension of $\gota_k(f(X), f(p))$ with the dimension of $\HilbPSm_k(X,p)$ and $\HilbPGor_k(X,p)$.
In Subsection~\ref{sect_Veronese_embedding} we explain that it is easy to construct strongly $k$-regular algebraic
  morphisms --- the well known Veronese varieties of sufficiently high degree are strongly $k$-regular.
In Section~\ref{sect_real_case} we explain how to use these results to conclude analogous statements for real manifolds.
Finally, in Appendix we explicitly calculate the dimension of  $\HilbPGor_k(X,p)$ for small values of $k$.

\begin{proof}[Proof of Theorem~\ref{thm_k_regular_to_space_of_dim_same_as_areole}]
   If $k=1$, then there is nothing to prove, so we may assume $f$ is one-to-one.
   Let $X_0$ be the smooth locus of the image $f(X)$ treated as a subset of $X$.
   Consider the partial compactification $\overline{C}_k(X_0) = \HilbSm_k(X_0)$
      of $C_k(X_0)$ as in Subsection~\ref{sect_partial_compactification}.
   Since  $\gota_k(f(X), p) = \gota_k(f(X_0), p)$, by Bertini theorem  \cite[Thm~I.6.3. 1a), 1b)]{jouanolou_Bertini}
      a general linear subspace $H \subset \CC\PP^N$
      of dimension $N- N'-1$ avoids  $\gota_k(f(X_0), p)$.
   Thus by Lemma~\ref{lem_projection_from_H_disjoint_with_areole}
      there exists an affinely $k$-regular continuous map $\CC^m \to \CC^{N'}$.
\end{proof}

\subsection{Areole and punctual Hilbert scheme}\label{sect_areole_and_punctual_Hilb}
The results in this subsection are independent of the base field $\FF$.

As motivated by the previous section, it is desirable to calculate, or at least bound the dimension of
   $\gota_k(X, p)$ for a subvariety $X \subset \FF\PP^N$ and $p \in X$ a smooth point.
By definition, $\gota_k(X, p)$ is parametrized by a family of (projective) linear subspaces of dimension $k-1$
   over $\HilbPSm_k(X,p)$.
Thus we immediately deduce:
\begin{equation}\label{equ_bound_on_dimension_of_areole_vs_smoothable}
  \dim \gota_k(X, p) \le \dim \HilbPSm_k(X,p) + k-1.
\end{equation}

Therefore, we first determine the dimensions of $\HilbPSm_k(X,p)$.
Note that this dimension does not depend on the embedding of $X$ into the projective space.
Moreover, $X$ can be replaced (for instance) by any Zariski open neighbourhood of $p$ in $X$,
  and in fact by any smooth variety of the same
  dimension, as mentioned at the end of Subsection~\ref{sect_punctual}.
The lower bound in the following Lemma relies on standard material reviewed in Appendix.
  
\begin{lema}\label{lem_bounds_on_dim_HilbPSm}
   In the setting as above, if $k\ge 2$, then we have:
   \[
      (\dim X -1)(k-1) \le \dim \HilbPSm_k(X,p) \le (k-1)\dim X -1
   \]
\end{lema}
\begin{proof}
   By \eqref{equ_dim_Hilb_R_and_Hilb_C} and the discussion in Subsection~\ref{sect_generalities_on_deformations},
      without loss of generality we may assume that the base field $\FF$ is algebraically closed.
      
   To obtain the lower bound $\dim \HilbPSm_k(X,p) \ge (\dim X -1)(k-1)$
     we construct a family of dimension  $(\dim X -1)(k-1)$
     of local smoothable schemes contained in $X$ and supported at $p$.
   This is obtained by the family of subschemes isomorphic to
     $\Spec \FF[t] / t^k$.
   The dimension of this family is calculated in
   Corollary~\ref{ref:alignabledim:cor}.
   Note that $\Spec \FF[t] / t^k$ is Gorenstein,
     thus the same lower bound also applies to $\HilbPGor_k(X,p)$.

   To obtain the upper bound, suppose $X$ is smooth and  consider
   \[
      \ccH := \bigcup_{x \in X} \HilbPSm_k(X,x)  \subset \HilbSm_k(X).
   \]
   The locus $ \HilbSm_k(X)$ is Zariski-irreducible and $\ccH \subsetneqq \HilbSm_k(X)$ (for $k\ge 2$),
     hence
     \[
      \dim \ccH \le \dim \HilbSm_k(X) -1 = k \dim X -1.
     \]
  Since $X$ is smooth, the dimension of $\HilbPSm_k(X,x)$ does not depend on
  $x$ (see Section~\ref{sect_punctual}). Furthermore,  $\HilbPSm_k(X,x)$  are pairwise disjoint
    for various points $x \in X$,
    hence 
    \[
       \dim \ccH = \HilbPSm_k(X,p) + \dim X.
    \] 
  Combining the two inequalities we obtain the desired upper bound.
\end{proof}

Depending on the values of $k$ and $\dim X$, both extremes $(\dim X -1)(k-1)$ or $(k-1)\dim X -1$
   may be obtained as $\dim \HilbPSm_k(X,p)$, see Theorem~\ref{ref:expecteddim:thm} and Example~\ref{ref:1551:example} in Appendix for more details.

On the other hand, we also have the following property:
If $R \subset X$ is a finite subscheme, then
\[
  \langle R \rangle = \bigcup_{\begin{array}{c}
                                 Q \subset R \\ Q \text{ is Gorenstein}
                               \end{array}}
                                \langle Q \rangle,
\]
see \cite[Lemma~2.3]{nisiabu_jabu_cactus}.
Thus $\gota_k(X,p)$ is contained in the variety swept out
   by linear spans of finite local Gorenstein schemes of length at most $k$ supported at $p$
   and so we obtain the following bound:
\begin{equation}\label{equ_bound_on_dimension_of_areole_vs_gorenstein}
  \dim \gota_k(X, p) \le \max\set{\dim \HilbPGor_{i}(X,p) + i-1 \mid 1 \le i \le k}.
\end{equation}

\begin{rem}
   Considering the schemes shorter than $k$ in Lemma~\ref{lem_map_to_grassmannian_and_no_closure}
     is merely a decoration to slightly facilitate the technical arguments.
   Yet here we restrict to punctual Gorenstein schemes and for these it is essential to consider also shorter schemes: for instance, there exist punctual Gorenstein schemes of length $4$, which are not contained in any punctual Gorenstein scheme of length $5$.
\end{rem}

In principle, we could restrict attention to a subset of $\HilbPGor_{i}(X,p)$
   consisting of those Gorenstein schemes that are contained in a smoothable scheme of length $k$.
Yet such description is not much helpful, as we are not able to provide better bounds on the dimension of these spaces.

\subsection{Veronese embedding as a \texorpdfstring{$k$}{k}-regular morphism}\label{sect_Veronese_embedding}
In this subsection we still work over the base field $\FF=\CC$.

Consider the $d$-th Veronese embedding:
\[
  v_d \colon \CC\PP^m\to \CC\PP^{{\binom{m+d}{m}}-1} = \PP (S^d \CC^{m+1})
\]
that is defined by a very ample line bundle $\mathcal{O}(d)$ on $\CC\p^m$.
Explicitly, in coordinates the map is given by all monomials of degree $d$ in $m+1$ variables.
The starting point of this subsection is a standard observation that a Veronese embedding of sufficiently high degree is
   $k$-regular:
\begin{lema}\label{lem_v_d_is_k_regular}
   Suppose $d \ge k-1$. Then the algebraic morphism $v_d$ is strongly projectively $k$-regular.
\end{lema}
See for instance~\cite[Lemma~2.3]{jabu_ginensky_landsberg_Eisenbuds_conjecture}.

The following corollary proves Theorems~\ref{thm_main_bound_intro_any_k} and \ref{thm_bound_by_hilbert_scheme_intro} in the complex case. The difference of $-1$ is due to considering affinely $k$-regular in the corollary and linearly $k$-regular maps in the introduction, see Lemma~\ref{lem_going_back_and_forth_between_regularities}.

\begin{cor}\label{cor_examples_of_k_regular_maps_over_C}
      Suppose $k\ge 2$. There exists an affinely $k$-regular continuous map $\CC^m \to \CC^N$,
    for certain $N \le \dim \HilbPSm_k(\CC\PP^m,p) + k-1 \le (k-1)(m+1) -1$ for any $p \in \CC\PP^m$.
\end{cor}
\begin{proof}
   By Lemma~\ref{lem_v_d_is_k_regular}, there exists a strongly projectively $k$-regular map,
      thus we may apply Theorem~\ref{thm_k_regular_to_space_of_dim_same_as_areole}:
      there exists an affinely $k$-regular continuous map $\CC^m \to \CC^{N}$,
      where $N=\dim \gota_k(v_d(\CC\PP^m), p)$.
   We also have $\dim \gota_k(v_d(\CC\PP^m), p) \le \dim \HilbPSm_k(\CC\PP^m,p) + k-1$
      by Inequality~\eqref{equ_bound_on_dimension_of_areole_vs_smoothable},
   and $\dim \HilbPSm_k(\CC\PP^m,p) + k-1 \le (k-1)(m+1) -1$ by Lemma~\ref{lem_bounds_on_dim_HilbPSm}.
\end{proof}

Next corollary proves Theorem~\ref{thm_bound_by_gorenstein_hilbert_scheme_intro} in the complex case.
\begin{cor}\label{cor_examples_of_k_regular_maps_over_C_from_gorenstein}
    Suppose $k\ge 2$. There exists an affinely $k$-regular continuous map $\CC^m \to \CC^N$,
       for some $N \le \max\set{\dim \HilbPGor_{i}(X,p) + i-1 \mid 1 \le i \le k}$.
\end{cor}
\begin{proof}
    Again, there exists an affinely $k$-regular continuous map $\CC^m \to \CC^{N}$,
      where $N=\dim \gota_k(v_d(\CC\PP^m), p)$.
   We have 
   \[
       \dim \gota_k(v_d(\CC\PP^m), p) \le \max\set{\dim \HilbPGor_{i}(X,p) + i-1 \mid 1 \le i \le k}
   \]
   by Inequality~\eqref{equ_bound_on_dimension_of_areole_vs_gorenstein}.
\end{proof}
The main point of the restriction to the Gorenstein schemes is that
   we have more control on the dimension of $\HilbPGor_{i}(X,p)$.
We show this in Appendix, here we conclude the proof of Theorem~\ref{thm_main_bound_intro_small_k}.
\begin{proof}[Proof of Theorem~\ref{thm_main_bound_intro_small_k}, complex case]
   By Corollary~\ref{cor_examples_of_k_regular_maps_over_C_from_gorenstein} 
       there exists an affinely $k$-regular continuous map $\CC^m \to \CC^N$,
       for $N = \max\set{\dim \HilbPGor_{i}(X,p) + i-1 \mid 1 \le i \le k}$.
   If $m \le 2$, then  $\dim \HilbPGor_{i}(X,p) = (i-1)(m-1)$ by Corollary \ref{ref:alignabledim:cor}
       and Proposition~\ref{ref:briancon:obs}.
   If $k \le 9$, then also $\dim \HilbPGor_{i}(X,p) = (i-1)(m-1)$ by Theorem~\ref{ref:expecteddim:thm}.
   In both cases, $N= (i-1)m$, hence there exists a linearly $k$-regular map $\CC^m \to \CC^{(i-1)m+1}$ by
       Lemma~\ref{lem_going_back_and_forth_between_regularities}.
\end{proof}

\section{Real case and other manifolds}\label{sect_real_case}

\subsection{Real case}
In this section we observe that the constructions are also valid over real numbers.
That is, we prove real versions of Theorems~\ref{thm_main_bound_intro_any_k}, \ref{thm_main_bound_intro_small_k}, \ref{thm_bound_by_hilbert_scheme_intro}, and \ref{thm_bound_by_gorenstein_hilbert_scheme_intro}.

\begin{thm}
   Suppose $k\ge 2$. There exists an affinely $k$-regular continuous map $\RR^m \to \RR^N$,
    \begin{itemize}
      \item for certain $N \le \dim \HilbPSm_k(\CC\PP^m,p) + k-1 \le (k-1)(m+1) -1$ for any $p \in \CC\PP^m$, and
      \item for certain $N \le \max\set{\dim \HilbPGor_{i}(\CC\PP^m,p) + i-1 \mid 1 \le i \le k}$.
    \end{itemize}
\end{thm}
\begin{proof}
   First, let us recall the construction in the complex case.
   We start with the Veronese embedding of projective space, which is, in fact, defined over real numbers.
   Then we project from a general linear subspace.
   Since $\dim \Hilb_k(\CC\PP^m) = \dim \Hilb_k(\RR\PP^m)$ 
      (and analogously for the variants of Hilbert scheme,  see \eqref{equ_dim_Hilb_R_and_Hilb_C}),
      the dimension of the set of $\RR$-rational points in
      $\Hilb_k(\RR\PP^n)$ and its variants is at most the dimension of the complex analogue.
   Similarly, the dimension of $\gota_k(v_d(\RR\PP^m))$ is bounded from above by the dimension in the complex case.
   Therefore as the general linear subspace that is the centre of projection,
      we can pick one which is defined over real numbers.
   Thus the theorem follows.
\end{proof}

\subsection{Regular embeddings of hypersurfaces and low codimension manifolds}

Suppose $M$ is a real manifold of dimension $m$ embedded in $\RR^{m+1}$.
The $k$-regular maps $\RR^{m+1} \to \RR^N$ naturally restrict to $k$-regular maps $M \to \RR^N$.
As an example, one may think of $M = S^m$, the $m$-dimensional sphere.

\begin{exm}
   Let $M$ be as above, and $k \ge 3$,
   Then $M$ admits an affinely $k$-regular embedding into $\RR^N$,
   where
   \[
     N = \begin{cases}
            (m+1)(k-1)    & \text{if  $k \le 9$, or $m=1$} \\
            (m+2)(k-1) -1 & \text{if $10 \le k \le m +2$}\\
            (m+1)k -1 & \text{otherwise.}
         \end{cases}
   \]
\end{exm}

\begin{proof}
   The example restricts the $k$-regular maps from Theorem~\ref{thm_main_bound_intro_small_k} (first case) and
   Theorem~\ref{thm_main_bound_intro_any_k} (second case).
   In the last case, the standard transversality construction, bounding $N$ by the dimension of secant variety
      (see Lemma~\ref{lem_projection_from_H_disjoint_with_secant}, or \cite{kreg6}) is stronger.
\end{proof}

More generally, we may apply the same argument to any $m$-dimensional submanifold $M \subset \RR^{c+m}$ of codimension $c$.
\begin{prop}
   Let $M^m \subset \RR^{c+m}$ be as above, and $k \ge 3$, $c \ge 2$.
   Then $M$ admits an affinely $k$-regular embedding into $\RR^N$,
   where
   \[
     N = \begin{cases}
            (m+c)(k-1)    & \text{if  $k \le \min(9, \frac{m+c}{c-1})$} \\
            (m+c+1)(k-1) -1 & \text{if $10 \le k \le \frac{m +c +1}{c}$}\\
            (m+1)k -1 & \text{otherwise.}
         \end{cases}
   \]
\end{prop}

\begin{exm}
   Let $M^m \subset \RR^{m+2}$, and $k \ge 3$.
   Then $M$ admits an affinely $k$-regular embedding into $\RR^N$,
   where
   \[
     N = \begin{cases}
            (m+2)(k-1)    & \text{if  $k \le \min(9, m+2)$} \\
            (m+3)(k-1) -1 & \text{if $10 \le k \le \frac{m +2 +1}{2}$}\\
            (m+1)k -1 & \text{otherwise.}
         \end{cases}
   \]
\end{exm}

\subsection{Globally \texorpdfstring{$k$}{k}-regular maps from a projective space}

Suppose $\FF=\CC$.
Arbitrary algebraic morphisms between two projective spaces are obtained as projections from the Veronese embedding.
That is, suppose $f \colon \PP V\to \CC\PP^N$ is an algebraic morphism explicitly given by:
\[
 [x] \mapsto [f_0(x), f_1(x), \dotsc, f_N(x)]
\]
with $f_i$ homogeneous polynomials of degree $d$, and suppose the linear subspace $L \subset S^dV^*$
  spanned by $f_i$ has dimension $N+1$ (i.e.~the polynomials $f_i$ are linearly independent).
Then $f = \pi_L \circ v_d$, where $v_d$ is the degree $d$ Veronese map $v_d\colon \PP V \to \PP(S^d V)$,
  and $\pi_L \colon \PP(S^d V) \to \CC\PP^N$ is the linear projection with the kernel $L^{\perp} \subset S^d V$.

Note that our assumption $\dim L = N+1$ is not very restrictive.
If $\dim L < N+1$, then the image of $f$ is contained in a linear subspace, and $f = \iota \circ \pi_L \circ v_d$,
  where $\iota \colon \PP L^* \to \CC\PP^N$ is a linear embedding.

\begin{lema}\label{lem_k_regular_has_degree_at_least_k_minus_1}
  Suppose  an algebraic morphism $f\colon \CC\PP^m \to \CC\PP^N$ as above is projectively $k$-regular.
  If $m>0$, then the degree $d$ of polynomials $f_i$ satisfies $d\ge k-1$.
\end{lema}
\begin{proof}
  Pick $k$ distinct points on a line in $\CC\PP^1 \subset \CC\PP^m$.
  Their images are contained in the image of the line, and the linear span of the image of the line is at most $d$-dimensional.
  The linear span of these $k$ points is $k-1$ dimensional, hence $d \ge k-1$.
\end{proof}

The dimensions of the secant varieties of the Veronese embedding are explicitly known, due to a Alexander-Hirschowitz theorem.
\tw[Alexander-Hirschowitz \cite{AH,OttAH}]\label{thm:AH}
Consider the $k$-th secant variety $\sigma_k(v_d(\CC\PP^m))$ of the $d$-th Veronese embedding of the projective space. The dimension of the secant variety is always the expected one, equal to $\min(km+k-1, {\binom{m+d}{d}}-1)$, except the following cases:
\renewcommand{\theenumi}{(\alph{enumi})}
\renewcommand{\labelenumi}{\theenumi}
\begin{enumerate}
\item\label{item_case_AH_d=2} $d=2$, $2\leq k\leq m$;
\item\label{item_case_AH_m=2_d=4_k=5}  $n=2$, $d=4$, $k=5$;
\item $m=3$, $d=4$, $k=9$;
\item $m=4$, $d=3$, $k=7$;
\item $m=4$, $d=4$, $k=14$.
\end{enumerate}
\ktw

In the context of $k$-regularity and in the view of Lemma~\ref{lem_k_regular_has_degree_at_least_k_minus_1},
  we can assume $d+1 \ge k > 2$. (Existence of $2$-regular maps $\CC\PP^m \to \CC\PP^N$ is clear whenever $N \ge m$.)
With this restriction, the theorem implies that $\dim \sigma_k(v_d(\CC\PP^m)) = km+k-1$, unless:
\renewcommand{\theenumi}{(\roman{enumi})}
\renewcommand{\labelenumi}{\theenumi}
\begin{enumerate}
  \item\label{item_case_restricted_m=1} If $m=1$, $\frac{d+1}{2} < k \le d+1$, then $\dim \sigma_k(v_d(\CC\PP^1)) = d$.
  \item\label{item_case_restricted_m=2_d=2_k=3} If $m=2$, $d=3$, $k=4$, then $\dim \sigma_4(v_3(\CC\PP^2)) = 9$.
  \item\label{item_case_restricted_m=2_d=4_k=5} If $m=2$, $d=4$, $k=5$, then $\dim \sigma_5(v_4(\CC\PP^2)) = 13$.
  \item\label{item_case_restricted_d=2} If $d=2$, $k = 3$, $m \ge 2$,
                                                                   then $\dim \sigma_3(v_2(\CC\PP^m)) = 3m-1$.
\end{enumerate}
Case \ref{item_case_restricted_m=2_d=4_k=5} is the exceptional case \ref{item_case_AH_m=2_d=4_k=5} in Theorem~\ref{thm:AH}.
Case \ref{item_case_restricted_d=2} with $m \ge 3$ is the quadric case \ref{item_case_AH_d=2}.
The remaining cases arise from solving the inequality $km + k - 1 > {\binom{m + d}{d}} - 1$, together with $d-1 \ge k > 2$.

\wn\label{kreg}
   For $k>2$, there exists a strongly projectively $k$-regular algebraic morphism $\CC\PP^m\rightarrow\CC\PP^N$
     if and only if either $N \ge km+k-1$ or:
\begin{enumerate}
   \item $m=1$ and $N \ge k-1$, or
   \item $m=2$, $k=4$, and $N \ge 9$, or
   \item $m=2$, $k=5$, and $N \ge 13$, or
   \item $k=3$, and $N \ge 3m-1$.
\end{enumerate}
\kwn
\dow
First we prove the existence.
Let $d=k-1$ and consider the Veronese embedding $v_d \colon \CC\PP^m \to \CC\PP^{\binom{m+d}{m}-1}$.
If $N \ge \binom{m+d}{m}-1$ then let $f = v_d$, possibly followed by a linear embedding. It is $k$-regular.
Hence from now on assume $N < \binom{m+d}{m}-1$.

By Theorem~\ref{thm:AH} and the following discussion, $N \ge \dim \sigma_k(v_d(\CC\PP^m))$.
Thus a general linear subspace $L \subset \CC\PP^{\binom{m+d}{m}-1}$ of codimension $N-1$ avoids $\sigma_k(v_d(\CC\PP^m))$.
In particular, $L$ does not intersect the linear span of any smoothable scheme of length at most $k$.

Consider the projection $\pi_L\colon \CC\PP^{\binom{m+d}{m}-1} \dashrightarrow \CC\PP^N$ with the centre $L$.
Since $k\ge 3$, the linear space $L$ avoids the second secant variety,
    and $f=\pi_L \circ v_d$ is an embedding of algebraic variety.
    Thus all subschemes of the image are images of schemes in the domain.
Hence the linear span of any subscheme $f(R) \subset f(\CC\PP^{m})$ is the image of $\langle v_d(R)\rangle$ under the projection $\pi_L$. Now if $R$ is a smoothable scheme of length $k$, then $\langle v_d(R)\rangle \cap L = \emptyset$, and
  $\dim \langle f(R)\rangle = \dim \langle v_d(R)\rangle = k-1$, i.e.~$f$ is strongly $k$-regular.

Now assume that $f\colon \CC\PP^m\rightarrow\CC\PP^N$ is a strongly projectively $k$-regular algebraic morphism.
Then $f = \pi_L \circ v_d$ as in the introductory paragraphs of this subsection.
By Lemma~\ref{lem_k_regular_has_degree_at_least_k_minus_1}, we have $d \ge k-1$.
We can ignore the following linear embedding by reducing $N$ if necessary.
We claim that $L$ must avoid the $\sigma_k(v_d(\CC\PP^m))$. By Lemma~\ref{lem_map_to_grassmannian_and_no_closure},
this is equivalent to say that $L$ does not intersect any linear span of
  $\dim \langle v_d(R)\rangle$ for smoothable schemes $R \subset \CC\PP^m$ of length $k$.
Otherwise, $\dim \langle f(R)\rangle < k-1$, contrary to our assumption that $f$ is strongly (projectively) $k$-regular.

Thus $L$ does not intersect $\sigma_k(v_d(\CC\PP^m))$, hence $\dim L + \dim \sigma_k(v_d(\CC\PP^m)) < N$ (with a convention that $\dim L =-1$ if and only if $L = \emptyset$). In particular, $\dim \sigma_k(v_d(\CC\PP^m)) \ge N$, which is equivalent to the claim of the corollary.
\kdow

\appendix
\section{Alignable schemes and dimension of the locus of finite Gorenstein schemes}

In this appendix we analyse the Hilbert scheme $\HilbPGor_k(\FF\PP^m, p)$. As
mentioned in Section~\ref{sect_punctual}, this is equivalent to the study of
$\HilbPGor_k(X, p)$ for any projective manifold $X$.

In Corollary~\ref{ref:alignabledim:cor} we prove a lower bound $(k-1)(\dim X -
1)$ for its dimension and in Theorem~\ref{ref:expecteddim:thm} we prove that
for $k\leq 9$ the dimension is equal to the lower bound.

Analysing the Hilbert scheme, as below, is technical and requires some knowledge of commutative algebra,
especially the theory of Artin Gorenstein algebras, see \cite{jelisiejew_MSc} for an overview and references.

In this appendix, as in the whole article, $\FF$ denotes either $\RR$ or
$\CC$. Many results of the appendix are true in greater  generality. The
exception is Brian\c{c}on's alignability result
(Proposition~\ref{ref:briancon:obs}), formulated only for complex numbers.
For the most part of the appendix we additionally assume that $\FF$ is
algebraically closed ($\FF = \CC$). We relax this assumption in
Theorem~\ref{ref:expecteddim:thm}.

\subsection{Generalities on deformations of finite algebras}\label{sect_generalities_on_deformations}

One of the fundamental properties of the Hilbert scheme is its invariance under base change 
  \cite[5.1.5(5), p.~112]{fantechi_et_al_fundamental_ag}. 
In particular, if $\FF_1 \subset \FF_2$ is a field extension (such as $\RR
\subset \CC$) and $p\in \FF_1\PP^m$ is an $\FF_1$-point, then:
\begin{equation*}
  \begin{aligned}
    \Hilb_k(\FF_2 \PP^m) &=  \Hilb_k(\FF_1 \PP^m) \times_{\FF_1} \Spec \FF_2,\\
    \HilbP_k(\FF_2 \PP^m, p) &=  \HilbP_k(\FF_1 \PP^m, p) \times_{\FF_1} \Spec \FF_2,\\
    \HilbGor_k(\FF_2 \PP^m) &=  \HilbGor_k(\FF_1 \PP^m) \times_{\FF_1} \Spec \FF_2,\\
    \HilbPGor_k(\FF_2 \PP^m, p) &=  \HilbPGor_k(\FF_1 \PP^m, p) \times_{\FF_1} \Spec \FF_2.\\
  \end{aligned}
\end{equation*}
In this article we are interested in  the dimension of the Hilbert scheme,
thus we have for every $\RR$-point $p\in \RR\PP^m$:
\begin{equation}\label{equ_dim_Hilb_R_and_Hilb_C}
  \begin{aligned}
    \dim \Hilb_k(\CC\PP^m) &=  \dim \Hilb_k(\RR \PP^m),\\
    \dim \HilbP_k(\CC\PP^m,p) &=  \dim \HilbP_k(\RR \PP^m,p),\\
    \dim \HilbGor_k(\CC\PP^m) &=  \dim \HilbGor_k(\RR \PP^m),\\
    \dim \HilbPGor_k(\CC\PP^m,p) &=  \dim \HilbPGor_k(\RR \PP^m,p).\\
  \end{aligned}
\end{equation}
Therefore it is harmless for us to assume that the base field $\FF$ is algebraically closed, 
   even though this assumption is eventually relaxed in the main theorem. 

We need some basic notions from deformation theory.
Below  an \textbf{Artin algebra} $A$ is a finite dimensional $\FF$-algebra and
$\dim_{\FF} A$ is called the \textbf{length} of $A$ and denoted $\len A$.
\begin{df}
    A \textbf{family} of Artin algebras is a finite flat morphism $\Spec C\to \Spec
    B$, where $B$ is an integral domain and a finitely generated $\FF$-algebra.
    We say that an algebra $A$ is a fibre of this family if there is a closed point
    $b\in \Spec B$ such that $\Spec A$ is isomorphic to the fibre $\Spec C_b$,
      i.e.~if $b$ corresponds to the maximal ideal $\mathfrak{b} \subset B$, then $A \simeq C\otimes_{B} B/\mathfrak{b}$.
\end{df}

\def\dissk#1{\mathbb{D}^{#1}}%
Let $\dissk{m}$ denote the spectrum of power series ring in $m$ variables. It
is isomorphic to the formal neighbourhood of a point $p\in \FF\PP^m$.
Informally, one may think of $\dissk{m}$ as a (very small) analytic neighbourhood.
In particular, to embed a finite scheme $R$ into $\FF\PP^m$ as a scheme supported
  at a pre-chosen point $p$ is the same as to
embed $R$ into $\dissk{m}$.

\def\embdim{\operatorname{embdim}}%
\begin{df}\label{df:embdim}
    Let $R = \Spec A$ be a punctual scheme corresponding to a local Artin
    algebra $A$ with maximal ideal $\mathfrak{m}$.
    The \textbf{embedding dimension} of $R$, is $\dim_\FF
    \mathfrak{m}/\mathfrak{m}^2$. It is denoted $\embdim R$.
\end{df}

The importance of embedding dimension comes from the fact that $\embdim R\leq
e$ if and only if $R$ has an embedding into $\FF\PP^e$ if and only if $R$ has
an embedding into $\dissk{e}$.

\subsection{Invariance of codimension and alignable subschemes}

Now we will analyse the following question.
\begin{question}
Let $\Spec A  \simeq R \subseteq \FF\PP^m$ be a finite scheme of length $k$ supported at $p$. What
is the dimension of the family of schemes isomorphic to $R$ in
$\HilbP(\FF\PP^m, p)$?
\end{question}
To give a scheme $R' \subseteq \FF\PP^m$
supported at $p$ and isomorphic to $R$ is the same as to give $R \simeq R' \subseteq
\dissk{m}$~i.e.~a surjective homomorphism from a power series ring
$\FF[[\alpha_1, \ldots ,\alpha_m]]$ onto the algebra $A$. Such a homomorphism is given by
sending each $\alpha_i$ to an element of the maximal ideal $\mathfrak{m}_A$ of $A$
and it is surjective
for a general such choice. Thus we have $m\cdot \dim_{\FF}\mathfrak{m} = m(k-1)$
parameters. Moreover, two such homomorphisms are equivalent if they have the
same kernel, i.e.~differ by
an automorphism of $A$. Thus the dimension is equal to $D := m(k-1) -
\dim\operatorname{Aut}(A)$. In particular, the ``codimension'' $m(k-1) - D$ is
independent of $m$.

\def\familyR{\mathcal{R}}
The following Proposition~\ref{ref:codiminv} generalises the above
considerations to any natural family 
 $\mathcal{R} \subset \HilbP_k(\FF\PP^e, p)$.
We need one technical definition.
Let $\familyR \subseteq \HilbP_k(\FF\PP^e, p)$. 
We say that $\familyR$ is \textbf{closed under isomorphisms} if every $R'\in \HilbP_k(\FF\PP^e, p)$ isomorphic to a scheme $R \in \familyR$
    belongs to $\familyR$.
In other words, if $i, i' \colon R \hookrightarrow \FF \PP^e$ are two embeddings of a finite scheme $R$ with support at $p \in \FF\PP^e$ and $i(R)$ is in the family $\familyR$, 
    then also $i'(R)$ is in $\familyR$. 
We will also use the flag Hilbert schemes (see~\cite[IX.7,
p.~48]{Arbarello__Geometry_of_algebraic_curves} or
\cite[Section 4.5]{Sernesi__Deformations}) and multigraded Hilbert schemes
(see~\cite{Haiman__Multigraded_Hilbert_schemes}).
It should be noted, that we
only consider these constructions for a scheme \emph{finite} over $\FF$,
a very special case.
\def\familyRe{\mathcal{R}^e}
\def\familyRm{\mathcal{R}^m}
\def\familyRem{\mathcal{R}^{e,m}}
\begin{obs}[invariance of codimension]\label{ref:codiminv}
    Let $m\geq e$ and $\familyRe \subseteq \HilbP_k(\FF\PP^e, p_0)$ be a
    Zariski-constructible subset closed under isomorphisms.
    Consider the subset $\familyRm \subseteq \HilbP_k(\FF\PP^m, p)$ of all
    schemes isomorphic to a member of $\familyRe$.
    Then the family $\familyRm$ is a Zariski-constructible subset of
    $\HilbP_k(\FF\PP^m, p)$ and its dimension satisfies
    \[
        (k - 1)m - \dim \familyRm = (k-1)e - \dim \familyRe.
    \]
\end{obs}

\def\dissktrimmed#1{\underline{\mathbb{D}}^{#1}}%
\begin{proof}
    For technical reasons (to assure existence of Hilbert flag schemes) we
    consider $\dissktrimmed{n} := \Spec \FF[[\alpha_1, \ldots ,\alpha_n]]/(\alpha_1,
    \ldots ,\alpha_n)^{k}$ rather than $\dissk{n}$. Let $R$ be a finite scheme of
    length $k$. To give an embedding $R \subseteq \FF\PP^n$ with support $p$
    is the same as to give an embedding $R \subseteq \dissk{n}$ and it is the same as
    to give an embedding $R \subseteq \dissktrimmed{n}$. Therefore
    $\Hilb_k(\FF\PP^n, p)  \simeq \Hilb_k(\dissktrimmed{n})$ for every $n$ and
    $p$.

    \def\HilbFlag{\operatorname{HilbFlag}}%
    \def\Hilbdiske{\Hilb_k(\dissktrimmed{e})}%
    \def\Hilbdiskindisk{\Hilb(\dissktrimmed{e} \subseteq \dissktrimmed{m})}%
    \def\Hilbdiskn{\Hilb_k(\dissktrimmed{m})}%
    Consider the multigraded flag Hilbert scheme $\HilbFlag$
    parametrising pairs of closed immersions $R \subseteq \dissktrimmed{e} \subseteq
    \dissktrimmed{m}$. It has natural
    projections $\pi_1$, $\pi_2$, mapping $R \subseteq \dissktrimmed{e} \subseteq
    \dissktrimmed{m}$ to $\dissktrimmed{e} \subseteq \dissktrimmed{m}$ and $R \subseteq
    \dissktrimmed{m}$ respectively, see diagram below.
    \[
        \begin{tikzcd}
            {} & \arrow{ld}{\pi_1}\HilbFlag{} \arrow{rd}{\pi_2} & \\
            \Hilbdiskindisk & & \Hilbdiskn
        \end{tikzcd}
    \]
    Note that $\pi_i$ are proper.
    We will now prove that $\familyRm$ is Zariski-constructible.
    Consider the automorphism group $G$ of $\dissktrimmed{m}$. It acts
    naturally on $\Hilbdiskindisk$  and $\HilbFlag$, making the morphism $\pi_1$
    equivariant. 
    The ideal of a $\dissktrimmed{e} \subseteq \dissktrimmed{m}$ is given by
    $m-e$ order one elements of the power series ring, linearly independent modulo
    higher order operators.  Therefore the action
    of $G$ on $\Hilbdiskindisk$ is transitive: for any two $(m-e)$-tuples as
    above there exists an automorphism of $\dissktrimmed{m}$
    mapping elements of the first tuple to the elements of the other tuple.

    Fix an embedding $\dissktrimmed{e} \subseteq
    \dissktrimmed{m}$, and hence an inclusion $i:\familyRe \hookrightarrow
    \HilbFlag$. Let $\familyRem = G\cdot i(\familyRe)$, then $\familyRm =
    \pi_2(\familyRem)$ and so it is Zariski-constructible. Note that
    $\familyRem = \pi_2^{-1}(\familyRm)$.

    It remains to compute the dimension of $\familyRm$. Let us redraw the
    previous diagram:
    \[
        \begin{tikzcd}
            {} & \arrow{ld}{\pi_1|_{\familyRem}}\familyRem \arrow{rd}{\pi_2|_{\familyRem}} & \\
            \Hilbdiskindisk & & \familyRm
        \end{tikzcd}
    \]
    Note that $\pi_1|_{\familyRem}$ is surjective because $\pi_1$ is $G$-equivariant 
       and $G$ acts transitively on the scheme $\Hilbdiskindisk$. 
    Furthermore, $\pi_1|_{\familyRem}$ has
        fibres isomorphic to $\familyRe$ because $\familyRe$ is closed under isomorphisms.
    Thus we obtain 
    \[
        \dim \familyRem = \dim \familyRe + \dim \Hilbdiskindisk.
    \]
    It remains to calculate the dimensions of $\Hilbdiskindisk$ and the fibre
    of $\pi_2$.
    An immersion $\varphi:\dissktrimmed{e} \subseteq
    \dissktrimmed{m}$ corresponds to a surjection $\varphi^{*}:\FF[[\alpha_1,
    \ldots ,\alpha_m]]/(\alpha_1, \ldots ,\alpha_m)^{k}\to \FF[[\beta_1, \ldots ,\beta_e]]/(\beta_1,
    \ldots ,\beta_e)^{k}$.
    Such surjective morphisms are parametrised
    by the images of generators $\varphi^{*}(\alpha_1), \ldots, \varphi^{*}(\alpha_m)$ in the maximal ideal $(\beta_1, \ldots, \beta_e)$.
    In fact, a general choice of those images gives a surjection. Let $M =
    \dim_{\FF} (\beta_1, \ldots, \beta_e)$ be the dimension of the \emph{ideal}
    $(\beta_1, \ldots , \beta_e)$. Then we have $mM$ parameters for the choice of
    $\varphi^{*}(\alpha_1), \ldots ,\varphi^{*}(\alpha_m)$. Two choices are equivalent if they
    have the same kernel, so that they differ by an automorphism of
    $\FF[[\beta_1, \ldots ,\beta_e]]/(\beta_1, \ldots ,\beta_e)^{k}$. This automorphisms
    group is $eM$ dimensional, thus $\dim \Hilbdiskindisk = mM - eM = (m-e)M$.

    Similarly, we may consider the fibre $\pi_2^{-1}(R)$ over a point $R\in
    \familyRm$ corresponding to a subscheme $R \subseteq \dissktrimmed{m}$. As above,
    the possible $\dissktrimmed{e}\subseteq \dissktrimmed{m}$ are
    parametrised by fixing the images of $\varphi^{*}(\alpha_1), \ldots ,\varphi^{*}(\alpha_{m})$
    in $\FF[[\beta_1, \ldots ,\beta_e]]/(\beta_1, \ldots ,\beta_e)^{k}$. The difference is
    that we have to ensure $R \subseteq \dissktrimmed{e}$. Algebraically, the
    images $\varphi^{*}(\alpha_1), \ldots , \varphi^{*}(\alpha_m)$ need to lie in the ideal
    $I(R) \subseteq (\beta_1, \ldots , \beta_e)$. Since $\dim_{\FF} I(R) = M-(k-1)$, the
    fibre has dimension $(m-e)(M - k +
    1)$.

    In particular, $\pi_2$ is equidimensional, so that the
    dimension of $\familyRm$  is given by the formula:
    \[
        \dim \familyRm = \dim \familyRe + (m-e)M - (m-e)(M-k+1) = \dim \familyRe + (m-e)(k-1).\qedhere
    \]
\end{proof}
We will now define an important class of finite schemes, analysed by Iarrobino
in~\cite{iarrobino_deforming_algebras_paper}.
\begin{df}
    An Artin algebra $A$ is \textbf{aligned} if it is isomorphic to $\FF[t]/t^k$.
    An Artin algebra $A$ is \textbf{alignable} if it is a finite flat limit of
    aligned algebras, i.e.~there exists a family of algebras with fibre $A$
    and general fibre aligned.
    A finite scheme $R = \Spec A$
    is aligned (reps.~alignable) if and only if $A$ is an aligned (respectively, alignable) algebra.
\end{df}
In terms of the Hilbert scheme, we see that $R$ is alignable if and only if
the point in the Hilbert scheme corresponding to $R$ is in the closure of the set
of points corresponding to the aligned schemes.

    Proposition~\ref{ref:codiminv} gives immediately the dimension of the set
    of alignable subschemes of a punctual Hilbert scheme.
    \begin{cor}[dimension of aligned schemes]\label{ref:alignabledim:cor}
        The locus of points of $\HilbP_k(\FF\PP^m, p)$ corresponding to aligned schemes
        has dimension $(k-1)(m-1)$.  Therefore, also the locus of alignable
        schemes has dimension $(k-1)(m-1)$.
    \end{cor}
    \begin{proof}
        If $m = 1$, then there is a unique closed subscheme of
        $\FF\PP^m$ isomorphic to $\Spec \FF[t]/t^k$ and supported at $p$, thus the
        dimension is $0$ and the claim is satisfied.

        Now let $m$ be arbitrary.  By Proposition~\ref{ref:codiminv} the
        dimension $d$ from the statement satisfies $(k-1)m - d = (k-1)$,
        thus $d = (k-1)(m-1)$.
    \end{proof}

    \def\HilbGP{\HilbPGor_k(\FF\PP^m, p)}%
    \begin{df}\label{ref:def}
        The \textbf{expected dimension} of the (smoothable) Gorenstein punctual Hilbert scheme of length $k$
        subschemes of $\FF\PP^m$ is the dimension of the family of alignable subschemes,
        i.e.~$(k-1)(m-1)$, see Corollary~\ref{ref:alignabledim:cor}.

        If a Zariski-constructible subset of the punctual Hilbert scheme has
        dimension less or equal than $(k-1)(m-1)$, then we call it
        \textbf{negligible}. In particular, the set of alignable subschemes is
        negligible.
    \end{df}

The name \emph{negligible} is taylored for the purposes of this article.
The dimension of a negligible family $\familyR$ is at most the dimension of the family of alignable schemes, 
   hence $\familyR$ does not give us any new information for the bound \eqref{equ_bound_on_dimension_of_areole_vs_gorenstein}
   on the dimension of the areole variety. Hence, such family can be neglected,
   at least as far as only dimension count is concerned.
    \begin{rem}\label{ref:invariance_of_negligible}
        Suppose that for some $k$ and $e$ the Gorenstein punctual Hilbert scheme
        $\HilbPGor_k(\FF\PP^e)$ has expected dimension. 
        Fix any $m\geq e$.
        Then the set of Gorenstein schemes in $\FF\PP^m$ of length
        $k$ and embedding dimension at most $e$ is negligible.
        Indeed, Proposition~\ref{ref:codiminv}
        implies that for every $m \geq e$ the set of schemes
        in $\HilbGP$ with embedding dimension $e$ has dimension
        $(k-1)m - (k-1)e + (k-1)(e-1) = (k-1)(m-1)$.
    \end{rem}

    \begin{rem}
        Gorenstein schemes of length at
        most $9$ are all smoothable by the results of Casnati and
        Notari, see~\cite[Theorem~A]{cn09}. Since ultimately we want to
        analyse those, we will not emphasise
        smoothability. However, the reader should note
        that Definition~\ref{ref:def} is reasonable only with the smoothability
        assumption.
    \end{rem}

    In general, there exist Gorenstein non-alignable subschemes and non-negligible
    families, see Example~\ref{ref:1551:example} below. In the remaining part
    of the appendix we show that if the length $k$ is small enough, then all
    subschemes are negligible, thus $\HilbGP$ has the expected dimension.

    We begin with the following result of Brian\c{c}on:
    \begin{obs}[{\cite[Theorem V.3.2, p.~87]{Briancon_surface}}]\label{ref:briancon:obs}
    Let $R \subset \CC\PP^2$ be any finite local scheme (not necessarily
    Gorenstein). Then $R$ is alignable.
\end{obs}
Thus the set of schemes with embedding dimension at most two is negligible by
Remark~\ref{ref:invariance_of_negligible}.
To analyse other schemes, we need a few results from the theory of
finite Gorenstein algebras, which we summarize below.

\subsection{Hilbert functions}\label{sect_Hilbertfunction}

\def\DmmA{\mathfrak{m}}%
Let $(A, \DmmA, \FF)$ be a local Artin Gorenstein $\FF$-algebra. We will recall two
discrete invariants of $A$.
The \textbf{Hilbert function} $H_A$ is defined by $H_A(n) = \dim
\DmmA^n/\DmmA^{n+1}$.

The \textbf{socle degree} of $A$ is
the maximal $s$ such that $H_A(s) \ne 0$.
For the following, we recall that the Hilbert function $H = H_A$ admits a canonical 
symmetric decomposition, see~\cite[Chapter 1, Thm 1.5]{ia94}. If we write
\[H = (1, H(1), H(2), \ldots ,H(s), 0, 0,  \ldots ),\] where $H(s)$ is the last non-zero value, then there exist
canonically defined vectors $\Delta_0 = \Delta_{A, 0},  \ldots , \Delta_{s-2} = \Delta_{A, s-2}$
of length $s+1$ whose elements are non-negative integers and such that
\begin{enumerate}
    \item $\Delta_i(s+1-i+n) = 0$ for all $n \geq 0$,
    \item $\Delta_i(s-i-n) = \Delta_i(n)$ for all $0\leq n\leq s-i$,
    \item $H(n) = \sum_{i=0}^{s-2} \Delta_{i}(n)$ for all $n\leq s$.
    \item $\Delta_0$ is a Hilbert function of an Artin Gorenstein algebra of
        socle degree $s$, in
        particular $\Delta_0(n) > 0$ for all $n\leq s$.
\end{enumerate}
For a fixed length $k$ there exists only finitely many Hilbert functions and
their decompositions, so to prove that a family is negligible, we may consider each Hilbert
function separately, i.e.~assume that the Hilbert function and its
decomposition are fixed.

\subsection{Apolarity}\label{sect_apolarity}

\def\hook{\lrcorner}%
\def\annn#1#2{\operatorname{ann}_{#1}(#2)}%
The canonical references for this section are~\cite{ia94} and
\cite{iarrobino_kanev_book_Gorenstein_algebras}.

Let $S = \FF[[\alpha_1, \ldots ,\alpha_e]]$ be a power series ring in $e$ variables.
We interpret $\alpha_i\in S$ as a partial differential operator
$\frac{\partial}{\partial x_i}$ acting on a polynomial ring $P = \FF[x_1, \ldots
,x_e]$. We denote this action by $(-)\hook(-)$. For example, $\alpha_1 \alpha_2 \hook
x_1^2x_2^3 = 6x_1x_2^2$ and $\alpha_3^2\hook x_1x_2x_3 = 0$.

    Let $(A, \DmmA, \FF)$ be a local Artin Gorenstein $\FF$-algebra. Suppose
    that the embedding dimension of $A$ is at most $e$. We may write $A$ as a
    quotient $A = S/I$ for some ideal $I$.  It turns out that  the annihilator
    $\annn{I}{P} = \{f\in P\ |\ I \hook f = 0\}$ is crucial for analysing
    $A$. Below we list its properties.
    \begin{enumerate}
        \item $\dim_{\FF} \annn{I}{P} = \dim_{\FF} A$,
        \item $\annn{I}{P}$ is an $S$-submodule of $P$, which is
            principal~i.e.~$\annn{I}{P} = S\hook f$ for some $f\in P$. This
            $f$ is not unique; in fact $S\hook f = S\hook g$ where $g = s\hook
            f$ for any invertible element $s\in S$. For example, $S\hook
            x^2y^2 = S\hook (x^2y^2 + xy^2 + x + 1)$.
        \item Conversely, given $f\in S$ we may consider its annihilator
            $I := \annn{f}{S} \subseteq S$. This is an ideal of $S$ and the
            quotient $S/I$ is a local Artin Gorenstein $\FF$-algebra.
    \end{enumerate}

    Summarising the discussion above, we have the following theorem,
    originally due to Macaulay.
    \begin{thm}[{\cite[Theorem~21.6, Exercise~21.7]{Eisenbud}}]
        Let $S$ be a power series ring over $\FF$ in $e$ variables, acting on
        a polynomial ring $P$ as described above. We have a
        correspondence between Artin Gorenstein quotients of $S$ and
        polynomials $f\in P$:
        \begin{enumerate}
            \item to every $f\in P$ we may assign the Artin Gorenstein algebra $S/\annn{S}{f}$,
                called the \textbf{apolar algebra} of~$f$.
            \item to every Artin Gorenstein algebra $S/I$ we may assign the module
                $\annn{P}{I}$, which is a principal $S$-module and choose any generator $f\in
                \annn{P}{I}$. We call any such $f$ a \textbf{dual socle
                generator} of $S/I$.
        \end{enumerate}
    \end{thm}
    We will use this theorem to bound the number of possible quotients of $S$
    by bounding the number of possible polynomials in $P$ whose apolar
    algebras have small length. It is worth noting that the symmetric
    decomposition of the Hilbert function of an Artin Gorenstein algebra
    (defined in subsection~\ref{sect_Hilbertfunction}) is tightly connected
    with the number of partial derivatives of $f$, which is implicitly used in
    the proof of Lemma~\ref{ref:neglige4:lem}. See~e.g.~\cite{ia94} for
    details.

\subsection{Punctual Gorenstein Hilbert schemes \texorpdfstring{$\HilbPGor_k(\FF\PP^m, p)$}{} for
\texorpdfstring{$k\leq 9$}{k at most 9}.}
Every Artin Gorenstein algebra $A$ of socle degree one is an apolar algebra of
a linear form, thus it is aligned. Therefore the set of socle degree one
algebras is negligible. In the following Lemma~\ref{ref:negligequadrics:lem}
we extend this result to socle degree two.
\begin{lema}\label{ref:negligequadrics:lem}
    The set $\mathcal{H} \subseteq \HilbPGor_k(\FF\PP^m, p)$
    consisting of local Artin Gorenstein algebras of socle degree two is
    negligible.
\end{lema}
\begin{proof}
    All members of $\mathcal{H}$ have length $k$ and socle degree $2$, hence their Hilbert function is equal to $(1, k-2, 1)$.
    In particular, $k-2\le m$, so using Proposition~\ref{ref:codiminv} similarly as in
    Remark~\ref{ref:invariance_of_negligible} we may assume $k-2 =m$.
    Algebras from $\mathcal{H}$ have length $m+2$ and are parametrised by a set of dimension
    $\binom{m+2}{2} - (m+2)$, see \cite[Thm 1,
    p.~350]{iarrobino_compressed_artin}. Therefore, $\mathcal{H}$ is negligible if
    \[
         \binom{m+2}{2} - (m+2) \leq (m+1)(m-1),
    \]
    which is true for every $m\geq 1$.
\end{proof}

Let $A$ be an algebra of socle degree $s\geq 3$.
Then $\Delta_{A, s-2} = (0, q, 0)$,
where $\Delta_{A, \bullet}$ were defined in Subsection~\ref{sect_Hilbertfunction}.
In the following we investigate the case when $q> 0$.
\begin{lema}\label{ref:negligesquares:lem}
    Consider the set $\mathcal{H}(q) \subseteq \HilbPGor_k(\FF\PP^m, p)$
    consisting of local Artin Gorenstein algebras of any socle degree $s\geq 3$ satisfying
    $\Delta_{s-2} = (0, q, 0)$.
    If all Gorenstein schemes of length $k - q$ and embedding dimension at most $m$ are
    negligible, then $\mathcal{H}(q)$ is negligible.
\end{lema}

\begin{proof}
    We argue by induction on $q$. In the base case $q = 0$ there is nothing to
    prove.

    Take any scheme $\Spec A \in \mathcal{H}(q)$.
    By \cite[Prop~4.5]{cjn13} the algebra $A$ is isomorphic to the apolar algebra of 
       $F + x_1^2 + \ldots  + x_q^2$, where $F$ is a polynomial in
    variables different from $x_1, \ldots ,x_q$. By \cite[Example 5.12]{cjn13}
    this algebra is an embedded limit of algebras of the form $B \times \FF$, where $B$
    has the same socle degree as $A$ and satisfies $\Delta_{B, s-2} = (0,
    q-1, 0)$.
    By induction the set of schemes corresponding to such $B$ is negligible,
    i.e.~has dimension at most $((k-1)-1)(m-1)$.
    Therefore the set of schemes corresponding to $B\times \FF$ has dimension
    at most $(k-2)(m-1) + m = (k-1)(m-1) + 1$. Since $\mathcal{H}(q)$ lies on
    the border of this set, $\dim \mathcal{H}(q) \leq (k-1)(m-1) + 1 - 1 =
    (k-1)(m-1)$.
\end{proof}

\begin{lema}\label{ref:neglige4:lem}
    Let $k\leq 10$ and $\mathcal{R} \subseteq \HilbPGor_k(\FF\PP^m, p)$ be the
    subset of schemes corresponding to local Artin Gorenstein algebras of socle degree
    at most four. Then $\mathcal{R}$ is negligible.
\end{lema}

\begin{proof}
    The family $\mathcal{R}$ divides into finitely many families according to
    Hilbert function and its symmetric decomposition. Therefore we may assume
    these are fixed in $\mathcal{R}$. Thus we may speak about Hilbert
    function, socle degree etc.

    We begin with a series of reductions.
    By induction and Remark~\ref{ref:invariance_of_negligible}, we may assume that the claim is true for schemes with
    embedding dimension less that $m$.
    Let $s$ be the socle degree of any member of $\mathcal{R}$. By
    Lemma~\ref{ref:negligequadrics:lem} we may assume that $s\geq 3$.
    By Lemma~\ref{ref:negligesquares:lem}, we may assume that $\Delta_{s-2}
    = (0, 0, 0)$.
    If $s = 3$, elements of $\mathcal{R}$ are parametrized by a set
    of dimension $\binom{m+3}{3} - (2m+2)$, see
    \cite[Thm 1, p.~350]{iarrobino_compressed_artin}, and $10 \geq k = 2m+2$
    (ibid), so
    $m\leq 4$. Then we
    need to check that $(k-1)(m-1) = (2m+1)(m-1) \geq \binom{m+3}{3} - (2m+2)$ for all
    $m\leq 4$.

    Similarly, if $s = 4$, then the Hilbert function
    has decomposition of the form $(1, a, b, a, 1) + (0, c, c, 0)$, where $a,
    b > 0$.
    We see that $k = 2 + 2a + 2c + b \leq 10$, $m = H(1) = a + c$.
    Moreover $b\leq \binom{a+1}{2}$ and from the Macaulay's Growth Theorem
    (see e.g.~\cite[Section~2.5]{cjn13}) it
    follows that either $b > 2$ or $a = b = 1$ or $a = b = 2$.

    Such algebras are parametrized by 
    \begin{itemize}
      \item the choice of a quartic in $a$ variables, which gives at most dimension $\binom{a+3}{4}$,
      \item the choice of these $a$ variables out of the linear space of $a+c$ variables, which gives at most $ac$,
      \item and a choice of polynomial of degree $3$ in $a+c$ variables: $\binom{a+c+3}{3}$,
      \item minus the length: $2 + 2a + 2c + b$.
    \end{itemize}
    Finally we get a parameter set of dimension at most
    \begin{equation}\label{eq:dimforfour}
        \binom{a+3}{4} + ac + \binom{a+c+3}{3} - (2 + 2a + 2c + b)
    \end{equation}
    Now one needs to the check that for all $a, b, c$ such that $2 + 2a +
    2c + b \leq 10$ satisfying the constraints above, the number
    \eqref{eq:dimforfour} is not higher than $(k - 1)(m-1)$.
\end{proof}

\begin{lema}\label{ref:neglige:lem}
    Suppose $k \leq 9$. Then the whole Gorenstein punctual Hilbert scheme
      \[
        \HilbPGor_k(\FF\PP^m, p)
      \]
      is  negligible.
\end{lema}

\begin{proof}
    Let $\mathcal{R} := \HilbPGor_k(\FF\PP^m, p)$.
    As before, we may fix a Hilbert function $H$ with symmetric
    decomposition $\Delta$, and a socle degree $s$.
    By Proposition~\ref{ref:briancon:obs} we may assume that the embedding
    dimension is at least three.
    By Lemma~\ref{ref:negligesquares:lem} we may assume that
    $\Delta_{s-2} = (0, 0, 0)$.
    By Lemma~\ref{ref:neglige4:lem} we may assume $s > 4$.

    We will prove that no decomposition $\Delta$ satisfying all above constrains exists.

    Let $e_i := \Delta_{A, i}(1)$. Then $H(1) = \sum e_i$.
    Note that $\Delta_{0} = (1, e_1,  \ldots , e_1, 1)$ is a vector of
    length $s+1\geq 6$, thus its sum is at least $4 + 2e_1$.
    Note that by symmetry of $\Delta_{i}$ we have $e_i =
    \Delta_{i}(s-i-1)$ and since $s-i-1 > 1$, we have $\sum_j
    \Delta_{i}(j) \geq 2e_i$. Summing up
    \[k = \sum H = \sum_i
    \sum_j \Delta_{i}(j)\geq 4 + 2\sum_i e_i \geq 4 + 2\cdot 3 = 10.\]
    This contradicts the assumption $k\leq 9$.
\end{proof}

\begin{thm}[The Hilbert scheme has expected
    dimension]\label{ref:expecteddim:thm}
    Let $\FF = \RR$ or $\FF = \CC$.
    For $k\le 9$ the dimension of $\HilbGP$  is expected as defined in~\ref{ref:def},
    i.e.
    \[
       \dim \HilbGP = (k-1)(m-1).
    \]
\end{thm}

\begin{proof}
    Suppose first that $\FF$ is algebraically closed.
    By Lemma~\ref{ref:neglige:lem} all subschemes of length at most $9$ are
    negligible, which by definition means that the dimension is not larger
    that the expected one, thus it is the expected one by
    Corollary~\ref{ref:alignabledim:cor}.

    Now let $\FF$ be arbitrary and denote by $\bar{\FF}$ its algebraic closure.
    Then as discussed in \eqref{equ_dim_Hilb_R_and_Hilb_C} above we have
    $\HilbPGor_k(\FF\PP^m, p) \times_{\Spec \FF} \Spec
    \bar{\FF}  = \HilbPGor_k(\bar{\FF}\PP^m, \bar{p})$, where $\bar{p}$
    is unique point in the preimage of $p$ in $\bar{\FF}\PP^m$. In particular, $\dim
    \HilbPGor_k(\FF\PP^m, p)$ is at most the expected one, thus it is the
    expected one by Corollary~\ref{ref:alignabledim:cor}.
\end{proof}

\begin{exm}\label{ref:1551:example}
    The dimension of the locus of alignable subschemes in
    $\HilbPGor_{12}(\FF\PP^5, p)$ is
    $(12-1)(5-1) = 44$. This locus is Zariski-irreducible and its general member is, by definition,
    isomorphic to $\Spec \FF[t]/t^{12}$. The subset $Z$ of
    $\HilbPGor_{12}(\FF\PP^5, p)$ parametrising subschemes with
    Hilbert function $(1, 5, 5, 1)$ has dimension $\binom{5+3}{3} - 12 = 44$,
    see~\cite[Thm~I]{iarrobino_compressed_artin}, thus $Z$ is not contained
    in the locus of alignable algebras, i.e.~a general subscheme with Hilbert
    function $(1, 5, 5, 1)$ is not alignable.

    Similarly, the subset $Z$ of $\HilbPGor_{16}(\FF\PP^7, p)$
    parametrising subschemes with Hilbert function $(1, 7, 7, 1)$ has
    dimension $\binom{7+3}{3} - 16 = 104 = (16 - 1)\cdot 7$. It is worth
    noting that both subsets parametrise \emph{smoothable} algebras, by
    \cite[Theorem~1]{JJ1551} and \cite[Theorem~7.3]{bertone2012division}.
\end{exm}

\input{k_regularity_export_arxiv.bbl}

 \bibliographystyle{amsalpha}
\end{document}

%% file: k_regularity_export_arxiv.bbl
\newcommand{\etalchar}[1]{$^{#1}$}
\def\cprime{$'$}\def\cprime{$'$} \def\cprime{$'$}
\providecommand{\bysame}{\leavevmode\hbox to3em{\hrulefill}\thinspace}
\providecommand{\MR}{\relax\ifhmode\unskip\space\fi MR }
\providecommand{\MRhref}[2]{%
  \href{http://www.ams.org/mathscinet-getitem?mr=#1}{#2}
}
\providecommand{\href}[2]{#2}

%% file: k_regularity_export_arxiv.bbl
\begin{thebibliography}{BCLZ14}

\bibitem[ACG11]{Arbarello__Geometry_of_algebraic_curves}
Enrico Arbarello, Maurizio Cornalba, and Pillip~A. Griffiths, \emph{Geometry of
  algebraic curves. {V}olume {II}}, Grundlehren der Mathematischen
  Wissenschaften [Fundamental Principles of Mathematical Sciences], vol. 268,
  Springer, Heidelberg, 2011, With a contribution by Joseph Daniel Harris.
  \MR{2807457 (2012e:14059)}

\bibitem[AH95]{AH}
J.~Alexander and A.~Hirschowitz, \emph{{Polynomial interpolation in several
  variables}}, J. Algebraic Geom. \textbf{4} (1995), no.~2, 201--222.
  \MR{1311347 (96f:14065)}

\bibitem[BB14]{nisiabu_jabu_cactus}
Weronika Buczy{\'n}ska and Jaros{\l}aw Buczy{\'n}ski, \emph{Secant varieties to
  high degree {V}eronese reembeddings, catalecticant matrices and smoothable
  {G}orenstein schemes}, J. Algebraic Geom. \textbf{23} (2014), 63--90.

\bibitem[BCLZ14]{blagojevic_cohen_luck_ziegler_On_highly_regular_embeddings_2}
Pavle V.~M. Blagojevi{\'c}, Frederick~R. Cohen, Wolfgang L{\"u}ck, and
  G{\"u}nter~M. Ziegler, \emph{On highly regular embeddings, {II}}, arxiv:
  1410.6052, 2014.

\bibitem[BCR12]{bertone2012division}
Cristina Bertone, Francesca Cioffi, and Margherita Roggero, \emph{{A} division
  algorithm in an affine framework for flat families covering {H}ilbert
  schemes}, arXiv preprint arXiv:1211.7264 (2012).

\bibitem[BGI11]{bernardi_gimigliano_ida}
Alessandra Bernardi, Alessandro Gimigliano, and Monica Id{\`a}, \emph{Computing
  symmetric rank for symmetric tensors}, J. Symbolic Comput. \textbf{46}
  (2011), no.~1, 34--53. \MR{2736357}

\bibitem[BGL13]{jabu_ginensky_landsberg_Eisenbuds_conjecture}
Jaros{\l}aw Buczy{\'n}ski, Adam Ginensky, and J.~M. Landsberg,
  \emph{Determinantal equations for secant varieties and the
  {E}isenbud-{K}oh-{S}tillman conjecture}, J. Lond. Math. Soc. (2) \textbf{88}
  (2013), no.~1, 1--24. \MR{3092255}

\bibitem[BLZ13]{Ziegler}
Pavle V.~M. Blagojevi{\'c}, Wolfgang L{\"u}ck, and G{\"u}nter~M. Ziegler,
  \emph{On highly regular embeddings}, arXiv preprint arXiv:1305.7483, to
  appear in Trans. AMS (2013).

\bibitem[BO08]{OttAH}
Maria~Chiara Brambilla and Giorgio Ottaviani, \emph{{On the
  {A}lexander-{H}irschowitz theorem}}, J. Pure Appl. Algebra \textbf{212}
  (2008), no.~5, 1229--1251. \MR{2387598 (2008m:14104)}

\bibitem[Bor57]{Borsuk}
K.~Borsuk, \emph{{On the {$k$}-independent subsets of the {E}uclidean space and
  of the {H}ilbert space}}, Bull. Acad. Polon. Sci. Cl. III. \textbf{5} (1957),
  351--356, XXIX. \MR{0088710 (19,567d)}

\bibitem[Bri77]{Briancon_surface}
Jo{\"e}l Brian{\c{c}}on, \emph{Description de
  {$H{\mathrm{ilb}}\sp{n}C\{x,y\}$}}, Invent. Math. \textbf{41} (1977), no.~1,
  45--89. \MR{0457432 (56 \#15637)}

\bibitem[CH78]{kreg2}
Frederick~R. Cohen and David Handel, \emph{{$k$}-regular embeddings of the
  plane}, Proc. Amer. Math. Soc. \textbf{72} (1978), no.~1, 201--204.
  \MR{524347 (80e:57033)}

\bibitem[Chi79]{kreg1}
Michael~E. Chisholm, \emph{{{$k$}-regular mappings of {$2^{n}$}-dimensional
  {E}uclidean space}}, Proc. Amer. Math. Soc. \textbf{74} (1979), no.~1,
  187--190. \MR{521896 (82h:55022)}

\bibitem[CJN15]{cjn13}
Gianfranco Casnati, Joachim Jelisiejew, and Roberto Notari,
  \emph{Irreducibility of the {G}orenstein loci of {H}ilbert schemes via ray
  families}, Algebra Number Theory \textbf{9} (2015), no.~7, 1525--1570.
  \MR{3404648}

\bibitem[CN09]{cn09}
Gianfranco Casnati and Roberto Notari, \emph{On the {G}orenstein locus of some
  punctual {H}ilbert schemes}, J. Pure Appl. Algebra \textbf{213} (2009),
  no.~11, 2055--2074. \MR{2533305 (2010g:14003)}

\bibitem[Eis95]{Eisenbud}
David Eisenbud, \emph{{Commutative algebra}}, {Graduate Texts in Mathematics},
  vol. 150, Springer-Verlag, New York, 1995, With a view toward algebraic
  geometry. \MR{1322960 (97a:13001)}

\bibitem[FGI{\etalchar{+}}05]{fantechi_et_al_fundamental_ag}
Barbara Fantechi, Lothar G{\"o}ttsche, Luc Illusie, Steven~L. Kleiman, Nitin
  Nitsure, and Angelo Vistoli, \emph{Fundamental algebraic geometry},
  Mathematical Surveys and Monographs, vol. 123, American Mathematical Society,
  Providence, RI, 2005, Grothendieck's FGA explained. \MR{2222646
  (2007f:14001)}

\bibitem[G{\"o}t94]{Gottsche_Hilbert_schemes_and_Betti_numbers}
Lothar G{\"o}ttsche, \emph{Hilbert schemes of zero-dimensional subschemes of
  smooth varieties}, Lecture Notes in Mathematics, vol. 1572, Springer-Verlag,
  Berlin, 1994. \MR{1312161 (96e:14001)}

\bibitem[Gra83]{Granger}
Michel Granger, \emph{G\'eom\'etrie des sch\'emas de {H}ilbert ponctuels},
  M\'em. Soc. Math. France (N.S.) (1983), no.~8, 84. \MR{716783 (86d:14004)}

\bibitem[Haa17]{haar__kreg}
Alfred Haar, \emph{Die {M}inkowskische {G}eometrie und die {A}nn\"aherung an
  stetige {F}unktionen}, Math. Ann. \textbf{78} (1917), no.~1, 294--311.
  \MR{1511900}

\bibitem[Han80]{kreg3}
David Handel, \emph{{Some existence and nonexistence theorems for {$k$}-regular
  maps}}, Fund. Math. \textbf{109} (1980), no.~3, 229--233. \MR{597069
  (82f:57018)}

\bibitem[Han96]{kreg4}
\bysame, \emph{{{$2k$}-regular maps on smooth manifolds}}, Proc. Amer. Math.
  Soc. \textbf{124} (1996), no.~5, 1609--1613. \MR{1307524 (96g:57025)}

\bibitem[Har66]{Hartcon}
Robin Hartshorne, \emph{{Connectedness of the {H}ilbert scheme}}, Inst. Hautes
  {\'E}tudes Sci. Publ. Math. (1966), no.~29, 5--48. \MR{0213368 (35 \#4232)}

\bibitem[HS80]{kreg5}
David Handel and Jack Segal, \emph{{On {$k$}-regular embeddings of spaces in
  {E}uclidean space}}, Fund. Math. \textbf{106} (1980), no.~3, 231--237.
  \MR{584495 (81h:57005)}

\bibitem[HS04]{Haiman__Multigraded_Hilbert_schemes}
Mark Haiman and Bernd Sturmfels, \emph{Multigraded {H}ilbert schemes}, J.
  Algebraic Geom. \textbf{13} (2004), no.~4, 725--769. \MR{2073194
  (2005d:14006)}

\bibitem[Iar72]{iarrobino_punctual_Hilbert_schemes}
Anthony Iarrobino, \emph{Punctual {H}ilbert schemes}, Bull. Amer. Math. Soc.
  \textbf{78} (1972), 819--823. \MR{0308120 (46 \#7235)}

\bibitem[Iar83]{iarrobino_deforming_algebras_paper}
\bysame, \emph{Deforming complete intersection {A}rtin algebras. {A}ppendix:
  {H}ilbert function of {$\mathbb{C}[x,\,y]/I$}}, Singularities, {P}art 1
  ({A}rcata, {C}alif., 1981), Proc. Sympos. Pure Math., vol.~40, Amer. Math.
  Soc., Providence, R.I., 1983, pp.~593--608. \MR{713096 (84j:14010)}

\bibitem[Iar84]{iarrobino_compressed_artin}
\bysame, \emph{Compressed algebras: {A}rtin algebras having given socle degrees
  and maximal length}, Trans. Amer. Math. Soc. \textbf{285} (1984), no.~1,
  337--378. \MR{748843 (85j:13030)}

\bibitem[Iar94]{ia94}
\bysame, \emph{Associated graded algebra of a {G}orenstein {A}rtin algebra},
  Mem. Amer. Math. Soc. \textbf{107} (1994), no.~514, viii+115. \MR{1184062
  (94f:13009)}

\bibitem[IK99]{iarrobino_kanev_book_Gorenstein_algebras}
Anthony Iarrobino and Vassil Kanev, \emph{Power sums, {G}orenstein algebras,
  and determinantal loci}, Lecture Notes in Mathematics, vol. 1721,
  Springer-Verlag, Berlin, 1999, Appendix C by Iarrobino and Steven L. Kleiman.
  \MR{1735271 (2001d:14056)}

\bibitem[Jel13]{jelisiejew_MSc}
Joachim Jelisiejew, \emph{Deformations of zero-dimensional schemes and
  applications}, arXiv:1307.8108, 2013.

\bibitem[Jel14]{JJ1551}
\bysame, \emph{Local finite-dimensional {G}orenstein {\it k}-algebras having
  {H}ilbert function (1,5,5,1) are smoothable}, J. Algebra Appl. \textbf{13}
  (2014), no.~8, 1450056, 7. \MR{3225123}

\bibitem[Jou83]{jouanolou_Bertini}
Jean-Pierre Jouanolou, \emph{Th\'eor\`emes de {B}ertini et applications},
  Progress in Mathematics, vol.~42, Birkh\"auser Boston Inc., Boston, MA, 1983.
  \MR{725671 (86b:13007)}

\bibitem[Kol48]{kolmogorov__kreg}
Andrey Kolmogorov, \emph{A remark on {P}. {L}. {C}hebyshev polynomials
  deviating least from zero}, Uspekhi Mat. Nauk \textbf{3:1} (1948), 216--221.
  \MR{0025609}

\bibitem[MS05]{miller}
Ezra Miller and Bernd Sturmfels, \emph{Combinatorial commutative algebra}, vol.
  227, Springer, 2005.

\bibitem[Ser06]{Sernesi__Deformations}
Edoardo Sernesi, \emph{Deformations of algebraic schemes}, Grundlehren der
  Mathematischen Wissenschaften [Fundamental Principles of Mathematical
  Sciences], vol. 334, Springer-Verlag, Berlin, 2006. \MR{2247603
  (2008e:14011)}

\bibitem[She02]{interpolation__Shekhtman__transversality}
Boris Shekhtman, \emph{Interpolation by polynomials in several variables},
  Approximation theory, {X} ({S}t. {L}ouis, {MO}, 2001), Innov. Appl. Math.,
  Vanderbilt Univ. Press, Nashville, TN, 2002, pp.~367--372. \MR{1924870
  (2004d:41061)}

\bibitem[She04]{interpolation__Shekhtman__poly}
\bysame, \emph{Polynomial interpolation in {$R_3$}}, Comput. Math. Appl.
  \textbf{48} (2004), no.~9, 1299--1304. \MR{2108481 (2005h:41041)}

\bibitem[She09]{interpolation__Shekhtman__ideal}
\bysame, \emph{Ideal interpolation: translations to and from algebraic
  geometry}, Approximate commutative algebra, Texts Monogr. Symbol. Comput.,
  SpringerWienNewYork, Vienna, 2009, pp.~163--192. \MR{2641160 (2011f:13029)}

\bibitem[Str96]{Stromme_Intro_to_Hilbert}
Stein~Arild Str{\"o}mme, \emph{Elementary introduction to representable
  functors and {H}ilbert schemes}, Parameter spaces ({W}arsaw, 1994), Banach
  Center Publ., vol.~36, Polish Acad. Sci., Warsaw, 1996, pp.~179--198.
  \MR{1481486 (98m:14004)}

\bibitem[Vas92]{kreg6}
V.~A. Vassiliev, \emph{{Spaces of functions that interpolate at any
  {$k$}-points}}, Funktsional. Anal. i Prilozhen. \textbf{26} (1992), no.~3,
  72--74. \MR{1189026 (94a:58024)}

\bibitem[Wul99]{wulbert1999interpolation}
Daniel Wulbert, \emph{Interpolation at a few points}, Journal of approximation
  theory \textbf{96} (1999), no.~1, 139--148.

\end{thebibliography}
